\documentclass[twoside]{article}

\usepackage[accepted]{aistats2026}          % submission (anonymous)
\usepackage[english]{babel}
\usepackage{natbib}
\usepackage{adjustbox}
\usepackage{advdate}
\usepackage{algorithm}
\usepackage{algpseudocode}
\usepackage{amsmath}
\usepackage{amsthm}
\usepackage{amssymb}
\usepackage{physics}
\usepackage{booktabs}
\usepackage{braket}
\usepackage{cancel}
\usepackage{caption}
\usepackage{enumitem}
\usepackage{bbm}
\usepackage{mathrsfs}
\usepackage{empheq}
\usepackage{mathtools}

\usepackage{adjustbox}
\usepackage{xcolor}
\usepackage{graphicx}
\usepackage[colorlinks=true, allcolors=blue]{hyperref}
\usepackage{mdframed} % no framemethod
\usepackage{soul}
\usepackage{subcaption}
\usepackage{multirow}
\usepackage{booktabs}
\usepackage{pgffor} % for \foreach
\usepackage{makecell} % Required for line breaks in cells

\makeatletter
\newenvironment{aistatsboxedeq}[2][orange!6]{%
  \begin{mdframed}[hidealllines=true,backgroundcolor=#1,
    innerleftmargin=#2pt, innerrightmargin=#2pt,
    innertopmargin=6pt, innerbottommargin=6pt]
    \@fleqnfalse\@mathmargin\z@
    \begin{equation}
}{%
    \end{equation}
  \end{mdframed}
}
\makeatother
\usepackage{pgfplots}
\pgfplotsset{compat=1.18}
\usepgfplotslibrary{statistics}

\definecolor{wotPeach}{HTML}{E4936C}
\definecolor{sbirrPeach}{HTML}{EF975F}
\definecolor{appexPeach}{HTML}{F4A261}

\definecolor{wotBlue}{HTML}{6EA8DC}
\definecolor{sbirrBlue}{HTML}{74B1D8}
\definecolor{appexBlue}{HTML}{7AB8F1}

\def\PotLabs{Bohachevsky,Oakley--O'Hagan,Quadratic,Styblinski--Tang,Wavy plateau}

\pgfplotsset{
  myboxesBase/.style={
    width=\textwidth, height=0.30\textwidth,
    xtick={1,2,3,4,5},
    xticklabels=\PotLabs,
    x tick label style={align=center},
    enlarge x limits=0.08,
    ymajorgrids, grid style={line width=.2pt, draw=gray!30},
    tick label style={font=\small}, label style={font=\small},
    boxplot/box extend=0.18,
    legend to name=methodsLegend,
    legend columns=3, legend style={draw=none,/tikz/every even column/.style={column sep=8pt}},
  },
  myboxesMAE/.style={
    myboxesBase,
    ymin=0, ymax=1.8,
    ytick={0,0.2,0.4,0.6,0.8,1.0,1.2,1.4,1.6,1.8},
    ylabel={Normalized L1 drift error (pointwise)},
  },
  myboxesCOS/.style={
    myboxesBase,
    ymin=0.5, ymax=1.0,
    ytick={0.5,0.6,0.7,0.8,0.9,1.0},
    ylabel={Cosine similarity (pointwise)},
  },
}

\usetikzlibrary{patterns,patterns.meta,shapes.geometric} % <-- only this; no patterns.meta

\usepackage{etoolbox} % for \ifstrequal if you use potfmt
\usepackage{float}

\usepackage{fancyhdr}

\newcommand{\N}{\mathbb{N}}

\newcommand{\R} {\mathbb R}

\newcommand{\dXt}{\mathrm{d}X_t}
\newcommand{\dYt}{\mathrm{d}Y_t}
\newcommand{\dt}{\mathrm{d}t}
\newcommand{\dx}{\mathrm{d}x}
\newcommand{\dWt}{\mathrm{d}W_t}

\DeclareMathSymbol{\shortminus}{\mathbin}{AMSa}{"39}

\newcommand{\jkonetstar}{\texttt{JKOnet}$^\ast$}

\newcommand{\nappex}{\texttt{nn-APPEX}}
\newcommand{\appex}{\texttt{APPEX}}

\newcommand{\wot}{\texttt{WOT}}
\newcommand{\sbirr}{\texttt{SBIRR}}

\usepackage[utf8]{inputenc} % allow utf-8 input
\usepackage[T1]{fontenc}    % use 8-bit T1 fonts
\usepackage{url}            % simple URL typesetting
\usepackage{booktabs}       % professional-quality tables
\usepackage{amsfonts}       % blackboard math symbols
\usepackage{nicefrac}       % compact symbols for 1/2, etc.
\usepackage{microtype}      % microtypography
\usepackage[capitalize]{cleveref}
\crefname{figure}{Figure}{Figures}
\Crefname{figure}{Figure}{Figures}
\newtheorem{theorem}{Theorem}[section] 
\newtheorem{lemma}[theorem]{Lemma}   
\newtheorem{prop}[theorem]{Proposition} 
\newtheorem{ex}[theorem]{Example}    
\newtheorem{cor}[theorem]{Corollary}  
  
\newtheorem{defn}[theorem]{Definition}  

\begin{document}

\twocolumn[
\aistatstitle{Gradient-Flow SDEs Have Unique Transient Population Dynamics}
% Langevin SDEs have unique transient dynamics
% Identifiability of stochastic potential-driven (population) dynamics
% Gradient/potential-driven SDEs have unique transient dynamics

\vspace*{-0.8cm}

\aistatsauthor{Vincent Guan \And Joseph Janssen \And Nicolas Lanzetti}

\aistatsaddress{University of British Columbia \And University of British Columbia \And ETH Zürich}

\vspace*{-0.1cm}

\aistatsauthor{Antonio Terpin \And Geoffrey Schiebinger \And Elina Robeva}

\aistatsaddress{ETH Zürich \And University of British Columbia \And University of British Columbia}
]

\runningauthor{Guan, Janssen, Lanzetti, Terpin, Schiebinger, Robeva}

\begin{abstract}
% The overdamped Langevin SDE is the predominant stochastic model used for gene expression data.
Identifying the drift and diffusion of an SDE from its population dynamics is a notoriously challenging task. Researchers in machine learning and single-cell biology have only been able to prove a partial identifiability result: for potential-driven SDEs, the gradient-flow drift can be identified from temporal marginals if the Brownian diffusivity is already known. Existing methods therefore assume that the diffusivity is known a priori, despite it being unknown in practice. We dispel the need for this assumption by providing a complete characterization of identifiability: the gradient-flow drift and Brownian diffusivity are jointly identifiable from temporal marginals if and only if the process is observed outside of equilibrium. Given this fundamental result, we propose \nappex, the first Schr\"odinger Bridge–based inference method that can simultaneously learn the drift and diffusion of a gradient-flow SDE solely from observed marginals. Extensive experiments show that \nappex's ability to adjust its diffusion estimate enables accurate inference, while previous Schr\"odinger Bridge methods obtain biased drift estimates due to their assumed, and likely incorrect, diffusion.

% Based on extensive numerical experiments illustrating its superior performance over previous Schrödinger Bridge methods, we demonstrate that \nappex's ability to adjust its diffusion estimate is crucial for accurate drift estimation.

% , thus demonstrating the importance of learning the diffusivity for accurate drift estimation.
% As processes of interest generally obey this condition, our result removes the long-standing assumption that the diffusion must be known to identify the drift. 

% We then complement our main result with theory and experiments in the finite sample setting, which 

% study practical identifiability, in order to propose heuristics for optimal data collection.

% Finally, we investigate the practical identifiability of the SDE parameters and propose heuristics for optimal data collection strategies.
% We corroborate our theory with several numerical experiments, investigate the practical identifiability of the SDE parameters, and propose heuristics for optimal data collection strategies.
\end{abstract}

\section{INTRODUCTION}
Mathematical biologists conceptualize the evolving genetic profiles of cells, or \textit{developmental trajectories}, with Waddington's epigenetic landscape, which likens differentiating cells to marbles rolling along a surface \citep{waddington1935animal}. Cell development is therefore thought to be driven by the gradient of this unknown landscape. More broadly, gradient-driven dynamics  describe a host of real-world processes, such as contaminant flow within groundwater systems, electric current within an electromagnetic field, and molecular dynamics driven by interatomic potential energy. Since these processes exhibit stochasticity, they are often modeled by \emph{gradient-flow} stochastic differential equations (SDEs), such that the drift is given by $-\nabla \Psi(X_t)$, and stochasticity is introduced by Brownian motion $W_t$, with diffusivity $\sigma^2>0$
\citep{weinreb2018fundamental, lavenant2021towards,lelievre2016partial}:
\begin{equation}\label{eq:overdamped_langevin_SDE}
    \dXt{} = -\nabla \Psi(X_t) \dt{} + \sigma \dWt{}.
\end{equation}
 While an SDE's drift and diffusion can generally be inferred from trajectories \citep{nielsen2000parameter, bishwal2007parameter, browning2020identifiability,wang2024generator}, it is often only possible to observe population dynamics. For example, hydrogeochemical sensors can only detect plumes, rather than particle trajectories, and scRNA sequencing technologies destroy cells upon measurement \citep{trapnell2014dynamics}. This limits observations to temporal snapshots of the marginals $\bigl(p(\cdot, t)\bigr)_{t \in [0,T]}$, and raises the fundamental question:
\begin{mdframed}[hidealllines=true,backgroundcolor=gray!8]
\begin{center}
Q: \emph{Under what conditions are $-\nabla\Psi$ and $\sigma^2$ identifiable from marginals $\bigl(p(\cdot, t)\bigr)_{t \in [0,T]}$?}
\end{center}
\end{mdframed}

\begin{figure*}[t]
  \centering
  \includegraphics[width=\linewidth]{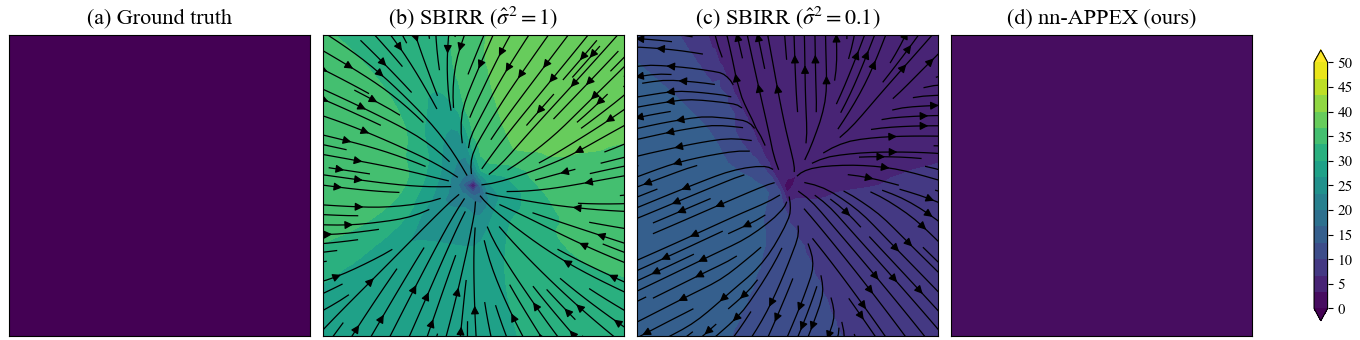}
  \caption{The true drift field (a) and estimated drift fields (b)-(d) are shown for the simple example of a Brownian motion, $\dXt{} = \sqrt{0.2} \dWt{}$. The current state-of-the-art Schr\"odinger Bridge method \sbirr\ \citep{shen2025multi, zhang2024joint} presumes prior knowledge ($\hat{\sigma}^2$) of the diffusivity $\sigma^2$ instead of inferring it from data. Figure 1(b) shows that it may wrongly infer a compressive drift force if $\hat{\sigma}^2 > \sigma^2$, while Figure 1(c) shows that it may wrongly infer an expanding drift force if $\hat{\sigma}^2 < \sigma^2$. Figure 1(d) shows that by iteratively learning the diffusion as well as the drift, our method \nappex{} can accurately infer drift without knowing diffusion a priori.}
\label{fig:misspecified_diffusion}
\end{figure*}
% Given its significance for understanding the biological processes governing disease, treatment, and development \citep{schiebinger2019optimal}, as well as its relevance for inferring population dynamics in other fields, like hydrology \citep{verwoerd2003theory}, electrodynamics \citep{roy2008langevin}, and human mobility \citep{persiianov2025learning},

% This problem has been widely studied in the last decade by researchers in mathematical biology 
% \vspace{-1mm}
Despite receiving significant attention from researchers in mathematical biology and machine learning
% and a flurry of machine learning inference methods developed in the last decade 
\citep{hashimoto2016learning, weinreb2018fundamental, neklyudov2023action, lavenant2021towards, yeo2021generative, zhang2021optimal, chizat2022trajectory,  bunne2022proximal, shen2025multi,terpin2024learning}, 
% no existing theory has established conditions under which it can be fully identified from its marginals. 
the most comprehensive identifiability results only prove that the drift $-\nabla \Psi$ can be identified from $\bigl(p(\cdot, t)\bigr)_{t \in [0,T]}$ if the diffusivity $\sigma^2$ is already known:

\begin{aistatsboxedeq}[orange!6]{6} % 6pt inner margins
\begin{aligned}
&\text{observe } \bigl(p(\cdot,t)\bigr)_{t\in[0,T]} \text{ and } \sigma^2 \text{ is known}\\
&\Longrightarrow\; -\nabla \Psi \text{ is identifiable}.
\label{eq:previous_identifiability}
\end{aligned}
\end{aistatsboxedeq}
\newpage
This result, popularized in the mathematical biology community by \citet{weinreb2018fundamental} and in the machine learning community by \citet{hashimoto2016learning}, has informed the predominant view that the only way to achieve principled inference is to assume that the diffusion is already known. As a result of this theory, virtually all inference methods have been developed to estimate $-\nabla \Psi$ given a known diffusivity $\sigma^2$ \citep{hashimoto2016learning, neklyudov2023action, lavenant2021towards, yeo2021generative, zhang2021optimal, chizat2022trajectory,  bunne2022proximal, shen2025multi,zhang2024joint}.
% sometimes directly citing non-identifiability to justify this assumption \citep{lavenant2021towards}. 
While the drift is often the principal object of interest, 
% there is little reason to believe that the diffusivity can be known a priori,   
the diffusion is also typically unknown, and contains important insights in its own right, such as the sensitivity of cell fates to initial conditions \citep{forrow2024consistent}. Perhaps most importantly, misspecified diffusion can significantly bias drift estimation. We provide a simple example in \cref{fig:misspecified_diffusion}, where the true landscape is flat, but overestimated diffusion introduces a sink and underestimated diffusion introduces a source. 
% The assumption of knowing or specifying diffusivity therefore significantly reduces the practicality of previous theory and the machine learning methods that were developed from these principles. 
% We provide a simple example in \cref{fig:misspecified_diffusion}, which shows how misspecified diffusion can bias drift estimation.

\paragraph{Contributions.} 
It has remained an open question whether it is possible and, if so, under which circumstances, \emph{both} the drift and diffusion of a gradient-flow SDE \eqref{eq:overdamped_langevin_SDE} can be identified from its marginals. Our first theoretical contribution is a complete solution to this problem. \cref{thm:identifiability} provides the \emph{necessary and sufficient conditions} for identifying the gradient flow drift and diffusivity from marginals:
% of when the drift and diffusion of the overdamped Langevin SDE \eqref{eq:overdamped_langevin_SDE} are \emph{jointly} identifiable from its marginals. Our main result, given in \cref{thm:identifiability}, provides the \emph{necessary and sufficient conditions} for identifiability:

\begin{mdframed}[hidealllines=true,backgroundcolor=blue!5,
                 innerleftmargin=6pt,innerrightmargin=6pt,
                 innertopmargin=6pt,innerbottommargin=6pt]
\begin{equation}
\begin{aligned}
&\text{observe non-stationary } \bigl(p(\cdot, t)\bigr)_{t \in [0,T]} \\
&\iff\; -\nabla \Psi \text{ and } \sigma^2 \text{ are identifiable}.
\end{aligned}
\end{equation}
\end{mdframed}
We also extend this result beyond the setting of continuous observation $t \in [0,T]$, by proving in \cref{cor:three_points} that observing three distinct marginals identifies the true SDE from any countable set of candidates, with probability $1$. Since our identifiability conditions are commonly observed in practice, our results provide:
\begin{enumerate}
    \item A theoretical foundation for jointly inferring drift and diffusion directly from marginals.
    \item A call to action to move beyond the conventional wisdom that diffusion should be specified a priori for principled inference from population dynamics.
\end{enumerate}
With our identifiability theory guaranteeing that joint inference is well-posed, we propose \nappex, the first Schr\"odinger Bridge (SB) based method capable of estimating both the gradient flow drift $-\nabla \Psi$ and the diffusivity $\sigma^2$. In particular, \nappex{} introduces a diffusion estimation step at each iteration while utilizing a neural network to flexibly infer the drift field. To assess the importance of \nappex's diffusion estimation step, we compare its performance against previous SB methods on simulated benchmark data \citep{terpin2024learning, persiianov2025learning} and on a single-cell dataset from human embryonic stem cells \citep{chu2016single,shen2025multi}. The results corroborate our identifiability theory, while also highlighting the importance of learning the correct diffusion for accurate drift estimation, as \nappex{} markedly outperforms previous SB methods.

\section{MATHEMATICAL SETUP}
\label{sec:math_background}
% We overview fundamental properties of the Langevin SDE and its population dynamics.

As is standard in stochastic analysis \citep{shen2025multi, berlinghieri2025oh}, we ensure that the gradient-flow SDE \eqref{eq:overdamped_langevin_SDE} is well-defined by assuming that the drift $-\nabla \Psi$ satisfies Lipschitz continuity and linear growth, $\|\nabla \Psi(x)\| \le K(1+\|x\|)$ for some $K>0$ \citep[Theorem 5.5]{oksendal2013stochastic}. 
% As is standard in stochastic analysis, to ensure the existence and uniqueness of SDE solutions for the gradient-flow SDE \eqref{eq:overdamped_langevin_SDE}, we assume that the drift $-\nabla \Psi$ satisfies Lipschitz continuity and linear growth \citep[Theorem 5.5]{oksendal2013stochastic}. 
% there exists a unique solution to the Langevin SDE \eqref{eq:overdamped_langevin_SDE}, driven by $-\nabla \Psi$, if the standard growth condition
% $\norm{\nabla \Psi(x)}\leq K(1+\norm{x})$
% $-x^T \nabla \Psi(x) \le K(1+\|x\|^2)$
Given these conditions, and an initial distribution $p(\cdot,0)$ with finite second moments,
% the gradient-flow SDE \eqref{eq:overdamped_langevin_SDE} is well-defined and 
the population dynamics are defined by the Fokker-Planck equation. For all $x \in \R^d, t \ge 0$,
\begin{align}
\begin{split}
    \frac{\partial p(x, t)}{\partial t}
    &=
    \mathcal{L}^*_{\shortminus\nabla\Psi, \sigma^2}(p(\cdot,t))(x)\\
    &\coloneqq 
    \nabla \cdot (p(x, t) \nabla \Psi(x)) + \frac{\sigma^2}{2} \Delta p(x, t).
    \label{eq:FP_eq_langevin}
    \end{split}
\end{align} 
While the Fokker-Planck equation \eqref{eq:FP_eq_langevin} may not be defined pointwise if the arguments are not sufficiently smooth, it can be defined in the weak distributional sense, such that $\mathcal{L}^*_{\shortminus\nabla\Psi, \sigma^2}$ is an operator on probability distributions (see \cref{subsec:weak_FP}). We therefore adopt the weak formulation and use standard notation by omitting $x$ \citep{bogachev2022fokker}, e.g., $\mathcal{L}^*_{\shortminus\nabla\Psi, \sigma^2}\bigl(p(\cdot,t)\bigr)$ rather than $\mathcal{L}^*_{\shortminus\nabla\Psi, \sigma^2}\bigl(p(\cdot,t)\bigr)(x)$.

% Given sufficient regularity on $p(\cdot,t)$ and $\Psi$, \eqref{eq:FP_eq_langevin} can be defined pointwise. More generally, if $p(\cdot,t)$ is a probability distribution, then the Fokker-Planck equation \eqref{eq:FP_eq_langevin} is defined weakly, such that 

We now rigorously define identifiability of gradient-flow SDEs, given continuous observation of their marginals. Intuitively, an SDE with parameters $(-\nabla \Psi_1, \sigma_1^2)$ is identifiable from its marginals if and only if no other SDE, with distinct parameters $(-\nabla \Psi_2, \sigma_2^2)\neq(-\nabla \Psi_1, \sigma_1^2)$, can produce the same probability flow $\frac{\partial p}{\partial t}$.
\begin{defn}[Identifiability]\label{def:non_iden}
The SDE \eqref{eq:overdamped_langevin_SDE}
with parameters $(-\nabla \Psi_1, \sigma_1^2)$ is identifiable from its marginals $\bigl(p(\cdot, t)\bigr)_{t \in [0,T]}$ if and only if, $\forall (-\nabla \Psi_2, \sigma_2^2)$,
\begin{align*}
&\mathcal{L}^*_{-\nabla \Psi_1, \sigma_1^2}\bigl(p(\cdot, t)\bigr) = \mathcal{L}^*_{-\nabla \Psi_2, \sigma_2^2}\bigl(p(\cdot, t)\bigr) \qquad \forall t \in [0,T]\\ &\implies (-\nabla \Psi_1, \sigma_1^2) = (-\nabla \Psi_2, \sigma_2^2),
\label{eq:non_iden_FP}
\end{align*}
where the first equality holds in the distributional sense and the second equality holds pointwise in $\R^d \times \R^+$.
\end{defn}
The following example \citep{lavenant2021towards, guan2024identifying} shows that gradient-flow SDEs ~\eqref{eq:overdamped_langevin_SDE} are not always identifiable from their marginals $\bigl(p(\cdot, t)\bigr)_{t \in [0,T]}$.
\begin{ex}[Non-identifiability at the stationary  distribution]
\label{ex:non_iden_stationary}
    Consider two SDEs with quadratic potential, i.e., Ornstein-Uhlenbeck processes,
\begin{align}
    \dXt{} &= -X_t \dt{} + \dWt{} \qquad\qquad\,\, X_0\sim p_0\\
    \dYt{} &= -10Y_t \dt{} + \sqrt{10}\dWt{} \qquad Y_0\sim p_0,
    \label{eq: non_iden_OU_1D}
\end{align}
with $p_0=\mathcal{N}(0, \frac{1}{2})$, i.e., Gaussian with mean $0$ and variance $\frac{1}{2}$. Since $p_0$ is the stationary distribution for both SDEs, then $X_t, Y_t \sim p(\cdot, t)=p_0$ for all $t \ge 0$, which makes the SDEs  non-identifiable from marginals.
% Thus, they are non-identifiable from $\bigl(p(\cdot,t)\bigr)_{t\in[0,T]}$.
\end{ex}
% The non-identifiability of population dynamics is a well-documented concern in the literature, since multiple SDEs may produce the same set of marginals $\bigl(p(\cdot, t)\bigr)_{t \in [0,T]}$.
% which would make the SDE~\eqref{eq:overdamped_langevin_SDE} intrinsically non-identifiable, regardless of the inference method.
% We note that non-identifiability of SDEs from marginals can arise for a multitude of reasons, including non-unique stationary distributions, undetectable rotations, and rank-degeneracy \citep{weinreb2018fundamental,lavenant2021towards,wang2024generator,guan2024identifying},
% which we discuss in \cref{sec: non_iden_examples}.

% \textcolor{red}{TODO}

% We note that if $p_0$

% It is well known that SDEs may be non-identifiable from observed temporal marginals \citep{weinreb2018fundamental,lavenant2021towards, shen2025multi, guan2024identifying, brogat2024learning}.

% \subsection{When is a Langevin SDE \emph{non}-identifiable from marginals?}
% \label{sec:non-ident}

\section{IDENTIFIABILITY RESULTS}
\label{sec:iden_theory}
% A classical example of non-identifiability is given by the Ornstein-Uhlenbeck (OU) processes, i.e. Langevin SDEs with quadratic potential: 
% \begin{align}
%     dX_t &= -X_t dt + dW_t \\
%     dX_t &= -10X_t dt + \sqrt{10}dW_t,
%     \label{eq: non_iden_OU_1D}
% \end{align}
% which are indistinguishable from marginals, when initialized at $p_0 \sim \mathcal{N}(0, \frac12)$. Although trajectories sampled from $dX_t = -X_t dt + dW_t$ have significantly less path variation compared to trajectories from $dX_t = -10X_t dt + \sqrt{10} dW_t$, both share the stationary distribution $\mathcal{N}(0, \frac12)$, which is determined by the drift to diffusivity ratio \citep{pavliotis2014stochastic}[Example 2.4]. 
% We overview additional examples of non-identifiability in Appendix \ref{sec: non_iden_examples},
% including undetectable rotations \citep{weinreb2018fundamental, shen2024learning, guan2024identifying} and rank-degenerate SDEs \citep{wang2024generator, guan2024identifying}. 

% For example, the Ornstein-Uhlenbeck (OU) SDE, $dX_t = -\lambda X_t dt + \sigma dW_t$, admits the stationary distribution $\mathcal{N}(0, \frac{\sigma^2}{2\lambda})$ \citep{pavliotis2014stochastic}

% Otherwise, \eqref{eq:overdamped_langevin_SDE} would not have a stationary distribution would not have a stationary distribution, since the drift would cause the marginals to escape to infinity. 
In this section, we present our main result in \cref{thm:identifiability}, which states that the gradient-flow SDE~\eqref{eq:overdamped_langevin_SDE} is identifiable from its marginals, $\bigl(p(\cdot, t)\bigr)_{t \in [0,T]}$, \emph{if and only if} it is observed outside of equilibrium.
% non-stationarity is a \emph{necessary} and \emph{sufficient} condition for identifiability of a Langevin SDE from its marginals.
We then show in \cref{cor:three_points} that three distinct marginals identify the true SDE with probability $1$ from any countable set of candidate SDEs.

We first prove that non-stationarity is a necessary condition for identifiability, by extending \cref{ex:non_iden_stationary} for arbitrary potentials.
% to show that non-identifiability persists for any gradient-flow SDE observed at equilibrium.

% which says that stationarity implies non-identifiability,

% the unique stationary distribution for a Langevin SDE can be explicitly computed if it exists \citep{pavliotis2014stochastic, lelievre2016partial, leimkuhler2016computation}, and is given by 
% For further details, see Appendix \ref{sec: langevin_stationary_appendix} and \citep[§4.5]{pavliotis2014stochastic}.
 % Thus, when the observed marginals only consist of the stationary distribution, the Langevin SDE~\eqref{eq:overdamped_langevin_SDE} is non-identifiable.
% We can therefore construct Langevin SDEs, which are non-identifiable when sampled from a shared equilibrium.
% \begin{prop}
% Let $\Psi(x) \in C^\infty(\R^d)$ be a confining potential, i.e.
% \[
% \lim_{|x| \to \infty} \Psi(x) = \infty \quad \text{ and } \quad
% \exp(-\beta \Psi(x)) \in L^1(\R^d) \qquad \forall \beta > 0.
% \]
% Then, the unique stationary distribution,  satisfying $\mathcal{L}_{\nabla \Psi, \sigma^2}(p_{\mathrm{eq}}(x)) = 0$, is the Gibbs distribution
% \begin{equation*}
% p_{\mathrm{eq}}(x) = \frac{1}{Z_{\Psi, \sigma^2}} \exp\left(-\frac{2\Psi(x)}{\sigma^2}\right), \qquad Z_{\Psi, \sigma^2} = \int_{\R^d} \exp\left(-\frac{2\Psi(x)}{\sigma^2}\right)dx.
% \end{equation*}
% \label{prop: unique_gibbs}
% \end{prop}

% \subsection{Necessary and sufficient conditions for structural identifiability}

\begin{prop}[Non-identifiability from equilibrium]
% [Stationary observations imply non-identifiability]
% For any smooth $\Psi$ and $\sigma, \alpha > 0$, 
If $p_\mathrm{eq}$ is a stationary distribution for the SDE \eqref{eq:overdamped_langevin_SDE}, then it is also a stationary distribution for the ``rescaled'' SDE
\begin{equation}
    \mathrm{d}X_t = -\alpha\nabla \Psi(X_t) \mathrm{d}t + \sqrt{\alpha}\sigma \mathrm{d}W_t,
\label{eq:rescaled}
\end{equation}
for any $\alpha>0$.
\label{prop:stationary_non_iden}
\end{prop}

\begin{proof}
% To fix the ideas, let us consider first the case of a confining potential $\Psi$ \citep[Definition 4.2]{pavliotis2014stochastic}. In this case, both \eqref{eq:overdamped_langevin_SDE} and \eqref{eq:rescaled}
% admit a unique stationary (Gibbs) distribution \citep[Proposition 4.2]{pavliotis2014stochastic}, and they coincide since $\frac{\alpha\Psi}{\alpha\sigma^2} = \frac{\Psi}{\sigma^2}$:
% \begin{equation}\label{eq: gibbs_dist}
%     p_{\mathrm{eq}}(x) = \frac{1}{Z_{\Psi, \sigma^2}} \exp\left(-\frac{2\Psi(x)}{\sigma^2}\right),
%     \qquad \text{where}\qquad
%     Z_{\Psi, \sigma^2} = \int_{\R^d} \exp\left(-\frac{2\Psi(x)}{\sigma^2}\right)\dx{}.
% \end{equation}
% The expression for the stationary distribution~\eqref{eq: gibbs_dist} suggests that stationarity only allows us, at best, to infer the ratio $\frac{\Psi}{\sigma^2}$ and not \emph{both} $\Psi$ and $\sigma^2$.

% This is also the key idea for the proof in the general case. 

If $p_\mathrm{eq}$ is stationary for $\mathcal{L}^*_{-\nabla\Psi, \sigma^2}$, then
\begin{align*}
    0 &= \mathcal{L}^*_{-\nabla\Psi, \sigma^2}(p_\mathrm{eq})\\
    &= \nabla \cdot (p_\mathrm{eq} \nabla \Psi)+ \frac{\sigma^2}{2} \Delta p_\mathrm{eq}
    \\\Leftrightarrow
    0 &= \alpha\left(\nabla \cdot (p_\mathrm{eq} \nabla \Psi + \frac{\sigma^2}{2} \Delta p_\mathrm{eq})\right)\\
    &= \nabla \cdot (p_\mathrm{eq}\nabla(\alpha\Psi)) + \frac{(\sqrt{\alpha}\sigma)^2}{2} \Delta p_\mathrm{eq}
    \\&= \mathcal{L}^*_{-\nabla(\alpha\Psi), \alpha\sigma^2}(p_\mathrm{eq}).
\end{align*}
Thus, given the initialization $p(\cdot,0)=p_\mathrm{eq}$, we would observe $p(\cdot,t)=p_\mathrm{eq}$ $\forall t \ge 0$ for both processes.
%$\mathcal{L}^*_{-\nabla\Psi, \sigma^2}$ and $\mathcal{L}_{-\nabla(\alpha\Psi), \alpha\sigma^2}$.
% In particular, if $p(x,0)=p_\mathrm{eq}(x)$ then  we have $p(x,t)=p_\mathrm{eq}(x)$ for both processes and the Langevin SDE~\eqref{eq:overdamped_langevin_SDE} is non-identifiable.
\end{proof}

% In particular, there are infinitely many Langevin SDEs sharing a given stationary distribution $p_{\mathrm{eq}} = \frac{1}{Z_{\Psi, \sigma^2}} \exp\left(-\frac{2\Psi(x)}{\sigma^2}\right)$. We note that their respective trajectories are time-rescaled analogues of one another, since \eqref{eq: time_rescaled_langevin} is equivalent to \eqref{eq:overdamped_langevin_SDE} following the substitution $\mathrm{d}t \to \alpha \mathrm{d}t$. 

% \subsection{When is a Langevin SDE identifiable from marginals?}

% We have shown that non-stationarity is a \emph{necessary} condition for identifiability.  
% %
This shows that the SDE~\eqref{eq:overdamped_langevin_SDE} is non-identifiable from its marginals $\bigl(p(\cdot, t)\bigr)_{t \in [0,T]}$ if we observe complete stationarity, $p(\cdot, t) = p(\cdot, 0)$ $\forall t \ge 0$. We now prove that this is the only source of non-identifiability.

% For this, we argue that if the marginals $\bigl(p(\cdot, t)\bigr)_{t \in [0,T]}$ of distinct Langevin SDEs evolve identically, then each marginal is in fact a stationary distribution to a \textit{residual} Langevin SDE. If the drift and diffusion of this SDE are time-homogeneous, then the uniqueness of the stationary distribution implies that the observed marginals are all equal, and hence stationary.

\begin{theorem}[Identifiability of gradient-flow SDEs]
\label{thm:identifiability} 
The SDE~\eqref{eq:overdamped_langevin_SDE} is non-identifiable from its marginals $\bigl(p(\cdot, t)\bigr)_{t \in [0,T]}$ if and only if $p(\cdot, t)=p(\cdot, 0)$ $\forall t \ge 0$.
\end{theorem}

\begin{proof}
\cref{prop:stationary_non_iden} proves that if $p(\cdot, t)=p(\cdot, 0)$ $\forall t \ge 0$ then the SDE is non-identifiable.

Now, we prove that non-identifiability implies that $p(\cdot, t)=p(\cdot, 0)$ $\forall t \ge 0$. Suppose that SDEs with parameters $(-\nabla \Psi_1, \sigma_1^2)$ and $(-\nabla \Psi_2, \sigma_2^2)$ produce the same marginals $\bigl(p(\cdot, t)\bigr)_{t \in [0,T]}$. Without loss of generality, we can assume that $\sigma_1^2 > \sigma_2^2$, because the previous partial result \eqref{eq:previous_identifiability} states that if $\sigma_1^2 = \sigma_2^2$, then $\nabla \Psi_1 = \nabla \Psi_2$ also follows \citep[Theorem 2.1]{lavenant2021towards}.
Then, from \cref{def:non_iden}, for all observed marginals $p(\cdot,t)$, 
\begin{align}
    \mathcal{L}^*_{-\nabla\Psi_1,\sigma_1^2}(p(\cdot,t))
    =
    \mathcal{L}^*_{-\nabla\Psi_2,\sigma_2^2}(p(\cdot,t)).
\end{align}
By the linearity of the Fokker-Planck operator, we have
\begin{align}
    0
    &=
    (\mathcal{L}^*_{\shortminus\nabla\Psi_1, \sigma_1^2}- \mathcal{L}^*_{\shortminus\nabla\Psi_2, \sigma_2^2})(p(\cdot, t))\\
    &=
    \nabla \cdot \bigl(p(\cdot, t)\nabla (\Psi_1-\Psi_2) \bigr) + \frac{\sigma_1^2-\sigma_2^2}{2}\Delta p(\cdot, t).
    \label{eq:fp-residual}
\end{align}
We note that the PDE \eqref{eq:fp-residual} is the Fokker-Planck equation, parametrized by the ``residual'' drift and diffusion. Then, to prove that $p(\cdot, t) = p(\cdot, 0)$ $\forall t \ge 0$, it suffices to show that there is at most one probability distribution, $p(\cdot, t) = \mu$, which solves \eqref{eq:fp-residual}, since each marginal would then coincide. Since the residual diffusivity $\sigma_1^2-\sigma_2^2>0$ is nondegenerate, and the residual drift $-\nabla (\Psi_1 -\Psi_2)$ is Lipschitz and obeys linear growth, it follows that the residual Fokker-Planck equation \eqref{eq:fp-residual} has at most one stationary distribution \citep[Theorem 4.1.6, Example 4.1.8]{bogachev2022fokker}.
\end{proof}

By reducing non-identifiability to the uniqueness of solutions to an elliptic PDE, our proof provides a fundamental insight about gradient-flow SDEs: \emph{observing non-identifiability between two processes is equivalent to observing a residual process at equilibrium}. % Indeed, there is a one-to-one correspondence between solutions to the Fokker-Planck PDE and the transition probabilities of the underlying Itô SDE, when $p(\cdot, t)$ is a probability measure\citep{figalli2008existence, schlichting2016fokker}.
% For two gradient-flow SDEs with identical population dynamics, we reduced this to the fact that a residual process, also a gradient-flow SDE, is in equilibrium.
If the SDE parameters are time-homogeneous, then each marginal $p(\cdot, t)$ coincides with the unique stationary distribution. Even if the parameters are time-inhomogeneous, non-identifiability is still characterized by equilibrium. However, equilibrium would not imply a static set of marginals, since equilibrium itself may change over time.
% Thus, observing transient marginals would no longer guarantee identifiability. 
We provide an example in \cref{ex: non_identifiability_time_inhomo}, where the stationary distributions of the residual process coincide with Brownian marginals $\mathcal{N}(0,t)$. In \cref{sec:time_inhomo}, we also prove sufficient conditions for identifiability in the time-inhomogeneous case, when the SDE parameters change at discrete times, a setting commonly observed in cell development \citep{monnier2012bayesian}.
% Indeed, there is a one-to-one correspondence between solutions to the Fokker-Planck PDE and the transition probabilities of the underlying Itô SDE, when $p(\cdot, t)$ is a probability measure\citep{figalli2008existence, schlichting2016fokker}.

% There is a one-to-one correspondence between solutions to the Fokker-Planck PDE and the transition probabilities of the underlying Itô SDE, when $p(\cdot, t)$ is a probability measure for all $t \in [0,T]$ \citep{figalli2008existence, schlichting2016fokker}. In particular, solutions to the PDE \eqref{eq: FP_eq} can always be interpreted as the marginals of the overdamped Langevin SDE \eqref{eq:overdamped_langevin_SDE}, as long as the solutions are probability measures.

% For this, we argue that if the marginals $\bigl(p(\cdot, t)\bigr)_{t \in [0,T]}$ of distinct Langevin SDEs evolve identically, then each marginal is in fact a stationary distribution to a \textit{residual} Langevin SDE. If the drift and diffusion of this SDE are time-homogeneous, then the uniqueness of the stationary distribution implies that the observed marginals are all equal, and hence stationary.

% We make a couple observations about the practical implications of our results. 
While \cref{prop:stationary_non_iden} shows that either the diffusivity $\sigma^2$ or the potential $\Psi$ needs to be known to disambiguate the true SDE from other SDEs with the same stationary distribution,~\cref{thm:identifiability} shows that this assumption is not required if the marginals are transient. This is important for real-world inference since the condition $p(\cdot, t)=p(\cdot, 0)$ $\forall t \ge 0$ is easily verifiable, and it is much more common to observe transient behaviour than it is to know an accurate a priori estimate of the diffusivity coefficient \citep{forrow2024consistent}. 
However, \cref{thm:identifiability} still relies on continuous observation, while we only observe a finite number of marginals in practice. We address this by showing that observing three distinct marginals is enough to identify the true SDE from any countable set of candidates, e.g. the set of SDEs that can be represented on a computer. To facilitate the proof, we assume that $\Psi \in C^\infty(\R^d)$ (see details in \cref{sec:additional_proofs}).

% However, we suspect that this condition can be removed, as we only use it for \cref{prop: decrease_free_energy}, which is a widely known result in stochastic analysis and statistical mechanics.

% and an infinite-data regime. However, in practice, one can only access a finite number of samples from a finite number of marginals. We empirically investigate the finite-data setting in \cref{sec:experiments}.

\begin{cor}[Identifiability from three marginals] 
% Let $\mathcal{S}$ be a \emph{countable} set of gradient-flow SDEs, and let $\{ t_i\}_{i=1}^3$ be a set of  measurement times, such that $t_i \sim \mathrm{Unif}[T_i, T_{i+1}]$, for some $0 \le T_1<T_2<T_3<T_4$. Then, if we observe distinct marginals $p(\cdot, t_1), p(\cdot, t_2), p(\cdot, t_3)$ from an SDE in $\mathcal{S}$, then, with probability $1$, it is the only SDE in $\mathcal{S}$ compatible with the marginals.
% SDE is unique.
% this is the unique SDE in $\mathcal{S}$ with the marginals $p(\cdot, t_1), p(\cdot, t_2), p(\cdot, t_3)$ at those times. 

Let $\mathcal{S}$ be a \emph{countable} set of gradient-flow SDEs with smooth potentials, which all share the same initial distribution $p(\cdot,0)$. If we observe distinct marginals $\{p(\cdot, t_i)\}_{i=1}^3$ 
% $p(\cdot, t_2), p(\cdot, t_3)$ 
from an SDE in $\mathcal{S}$, such that the measurement times $\{t_i\}_{i=1}^3$ are uniformly sampled, i.e., $t_i \sim \mathrm{Unif}[T_i, T_{i+1}]$, for some $0 < T_1<T_2<T_3<T_4$, then, with probability $1$, this is the only SDE in $\mathcal{S}$ with marginals $\{p(\cdot, t_i)\}_{i=1}^3$.

% Let $\mathcal{S}$ be a \emph{countable} set of gradient-flow SDEs

% and suppose that we sample three measurement times $\{t_i\}_{i=1}^3$ from $t_i \sim \mathrm{Unif}[T_i, T_{i+1}]$, for some $0 \le T_1 < T_2 < T_3 < T_4$. Then, with probability $1$, observing distinct marginals $\{p(\cdot, t_i)\}_{i=1}^3$ uniquely determines an SDE in $\mathcal{S}$.

% Then, with probability $1$, the obser

% Suppose we sample  independently from 

% $[T_i, T_{i+1}]_{i=1}^3$ be a set of time intervals. Then, if we uniformly sample , such that $t_i \sim \mathrm{Unif}[T_i, T_{i+1}]$, and observe  it follows that with probability $1$, it is the only SDE in $\mathcal{S}$ compatible with $\{p(\cdot, t_i)\}_{i=1}^3$.

% SDE is unique.
% this is the unique SDE in $\mathcal{S}$ with the marginals $p(\cdot, t_1), p(\cdot, t_2), p(\cdot, t_3)$ at those times. 
\label{cor:three_points}
\end{cor}
\begin{proof}
By the countability of $\mathcal{S}$, it suffices to prove that the event that any two distinct SDEs in $\mathcal{S}$ share the same marginals at each of the times $\{t_i\}_{i=1}^{3}$ has probability $0$. Denote their parameters by $(-\nabla \Psi_1, \sigma_1^2) \neq (-\nabla \Psi_2, \sigma_2^2)$ and their marginals by $\bigl( p(\cdot, t) \bigr)_{t \ge 0}$ and $\bigl( q(\cdot, t) \bigr)_{t \ge 0}$. Suppose for contradiction that if we sample $t_i \sim \mathrm{Unif}[T_i, T_{i+1}]$, then we observe $p(\cdot, t_i) = q(\cdot, t_i)$ $\forall i \in \{1,2,3\}$, with some nonzero probability.

% denote the pair of SDEs by their parameter sets $(-\nabla \Psi_1, \sigma_1^2)$ and $(-\nabla \Psi_2, \sigma_2^2)$, and 
 % By assumption, we have that , where $t_i \in [T_i, T_{i+1}]$.
% Without loss of generality, the three measurement times are sampled from corresponding subintervals, such that $t_i \sim [T_i, T_{i+1}]$, for some choice of $T_1<T_2<T_3<T_4$.
Since we sample $t_i \sim \mathrm{Unif}[T_i, T_{i+1}]$, then the set of coincidence times in each subinterval, $\mathcal{I}_i = \{t \in [T_i, T_{i+1}] \mid p(\cdot, t) = q(\cdot, t)\}$, must be infinite to ensure nonzero probability. By Bolzano-Weierstrass, we can construct the convergent sequence $t_1^{(n)}  \to t_1^* \in [T_1,T_2]$, using times $t_1^{(n)} \in \mathcal{I}_1$, and the convergent sequence $t_2^{(n)} \to t_2^*\in [T_2,T_3]$, using times $t_2^{(n)} \in \mathcal{I}_2$. At the limit points, $t_1^*$ and $t_2^*$, we have  $p(\cdot, t_i^*) = q(\cdot, t_i^*)$ and $\frac{\partial}{\partial_t} p(\cdot, t_i^*) = \frac{\partial}{\partial_t} q(\cdot, t_i^*)$ (see \cref{lemma:flow_marginals_match} and the proof of \citet[Corollary 2]{hashimoto2016learning}). Hence, 
\begin{align}
\label{eq:first_FP}
    \mathcal{L}^*_{\shortminus \nabla \Psi_1, \sigma_1^2}(p(\cdot, t_1^*)) &= \mathcal{L}^*_{\shortminus \nabla \Psi_2, \sigma_2^2}(p(\cdot, t_1^*)) \\
    \mathcal{L}^*_{\shortminus \nabla \Psi_1, \sigma_1^2}(p(\cdot, t_2^*)) &= \mathcal{L}^*_{\shortminus \nabla \Psi_2, \sigma_2^2}(p(\cdot, t_2^*)) 
\label{eq:second_FP}
\end{align}
% By the partial identifiability result \eqref{eq:previous_identifiability}, the two SDEs must have different diffusivities, i.e. $\sigma^2 = \sigma_1^2 - \sigma_2^2 > 0$, in order to have distinct parameters. 
It then follows from \cref{thm:identifiability} that \eqref{eq:first_FP} and \eqref{eq:second_FP} are both solved by at most one distribution $\mu$. Thus, $\mu = p(\cdot, t_1^*) = p(\cdot, t_2^*)$.
% which solves the residual PDE, i.e. $q(\cdot, t_i^*)=p(\cdot, t_i^*) = \mu$ for each $\{t_i^*\}_{i=1}^{3}$. 
% That is, if both processes were observed at the times $\{t_i^*\}_{i=1}^{3}$, then the same distribution would be repeated at three distinct times. 
However, by \cref{prop: decrease_free_energy}, the only marginal that can repeat at distinct times is the stationary Gibbs distribution $p_{\mathrm{eq}}$. This follows from a variation of Boltzmann's H-theorem, which states that the marginals of a gradient-flow SDE \eqref{eq:overdamped_langevin_SDE} decrease free energy, which is uniquely minimized at $p_{\mathrm{eq}}$ \citep{jordan1998variational,chafai2015boltzmann}. 
Both processes must therefore reach $p_{\mathrm{eq}}$ by $t_1^*$, such that  $p(\cdot, t) = q(\cdot, t) = p_{\mathrm{eq}}$ for all $t \ge t_1^*$. In particular, since $t_1^*$ is in the first subinterval, we have $p(\cdot, t_2) = p(\cdot, t_3) = p_{\mathrm{eq}}$, which contradicts the assumption that the marginals $\{p(\cdot, t_i)\}_{i=1}^3$ are distinct. 
% We conclude that the event has probability zero.
\end{proof}

\section{NN-APPEX: THREE-STAGE SB REFINEMENT}
\label{sec:method}
% \cref{sec:iden_theory} shows that it is feasible to identify the true SDE parameters $(-\nabla \Psi, \sigma^2)$ from observed temporal marginals. In particular, if given continuous time observations $\Ti = [0,T]$, 
We have shown in \cref{sec:iden_theory} that the full gradient-flow SDE is identifiable from $N \ge 3$ distinct marginals, $\bigl(p(\cdot, t_i)\bigr)_{i=0}^{N-1}$. Building on this, we propose Neural Network-based Alternating Projection Parameter Estimation from $X_0$ (\nappex). \appex{} searches for the true parameters $(-\nabla \Psi, \sigma^2)$ by alternating between trajectory inference, drift estimation, and diffusion estimation. \nappex{} improves upon \appex{} \citep{guan2024identifying} by replacing the linearly parameterized drift estimation step with a neural network, such that it can infer SDEs with nonlinear drift, e.g. arbitrary gradient-flow $-\nabla \Psi$.

% which can learn the full SDE, beyond those with linear drift. 
% The naming convention comes from the fact that the initial distribution $X_0$ determines theoretical identifiability (here, if $X_0 \not \sim p_{\mathrm{eq}}$).

% This added flexibility for modeling nonlinear drift comes from the fact that we estimate drift with a neural network rather than a pre-parameterized likelihood based on the assumption of linearity.

The three-stage optimization proceeds as follows. First, given a reference SDE $Q$, we solve a multi-marginal Schr\"odinger Bridge problem to infer a law on paths $P$. In particular, $P$ minimizes $\mathrm{KL}(P \| Q)$, while obeying the multi-marginal constraints, i.e., $P  \in \Pi\bigl(p(\cdot, t_i)_{i=0}^{N-1}\bigr)$. Second, from these inferred trajectories, we perform a maximum likelihood estimation (MLE) of the drift $-\nabla \Psi$. For generality and tractability, we optimize over neural network parameters $\theta$. Third, from the inferred trajectories and the estimated drift, we estimate the diffusivity $\sigma^2$ with a closed form MLE expression. We then use the new drift and diffusion estimates to update the reference SDE $Q$ for the next iteration. Formally, at iteration $k$, \nappex{} performs the steps
% inferring trajectories from marginals, given a reference SDE \eqref{eq: SB_traj_inf_k}, estimating drift from the inferred trajectories using a neural network \eqref{eq: drift_MLE_k}, and re-estimating diffusion 
% from the inferred trajectories, given the estimated drift \eqref{eq: diff_MLE_k}
\begin{align}
    \label{eq: SB_traj_inf_k}
    % P_{k} &= \arg\min_{P \in \Pi\bigl(p(\cdot, t_i)_{i=1}^N\bigr)} \mathrm{KL}(P \| Q_{\shortminus\nabla\Psi_{\theta_{k-1}}, \sigma^2_{k-1}}),
    P_{k} &= \mathop{\arg\min}_{P \in \Pi\left(p(\cdot, t_i)_{i=0}^{N-1}\right)} \mathrm{KL}(P \| Q_{\shortminus\nabla\Psi_{\theta_{k-1}}, \sigma^2_{k-1}}),\\
    \label{eq: drift_MLE_k}
    \theta_{k} &= \mathop{\arg\min}_{\theta} \mathrm{KL}(P_{k} \| Q_{\shortminus \nabla\Psi_{\theta}, \sigma^2_{k-1}}),\\
    \label{eq: diff_MLE_k}
    \sigma^{2}_{k} &= \mathop{\arg\min}_{\sigma^2} \mathrm{KL}(P_{k} \| Q_{\shortminus \nabla\Psi_{\theta_k}, \sigma^2}),
\end{align}
% The trajectory inference step \eqref{eq: SB_traj_inf_k} is the Schrödinger Bridge problem with respect to the reference SDE $Q$, parametrized by the previous iteration's estimated drift and diffusion. 
and we iterate until convergence, or until a maximum number of iterations is reached.  
We emphasize that \nappex{} does not require prior knowledge, since any gradient-flow SDE \eqref{eq:overdamped_langevin_SDE} can be used as the initial reference $Q_{-\nabla \Psi_{\theta_0}, \sigma^2_0}$. While it is known that the $\mathrm{KL}$ divergence on the space $C([0,T], \R^d)$ is infinite between SDEs with different diffusions \citep{shen2025multi}, we note that given finite observation times, $\{t_i\}_{i=0}^{N-1}$, we can instead evaluate the $\mathrm{KL}$ divergence on $C(\{t_i\}_{i=0}^{N-1}, \R^d)$, computed over the observed couplings $p_{t_i, t_{i+1}}$. As shown in \citet[Section 4]{guan2024identifying} and \cref{sec: supplement_sb_refinement}, this objective is finite for gradient-flow SDEs \eqref{eq:overdamped_langevin_SDE}, and thus enables principled three-stage optimization. We present pseudocode in \cref{alg:nappex}.

\begin{figure*}[t]
  \centering

  % ---- row of subfigures (a)–(c) ----
  \begin{subfigure}[b]{0.25\linewidth}
    \centering
    \includegraphics[width=\linewidth]{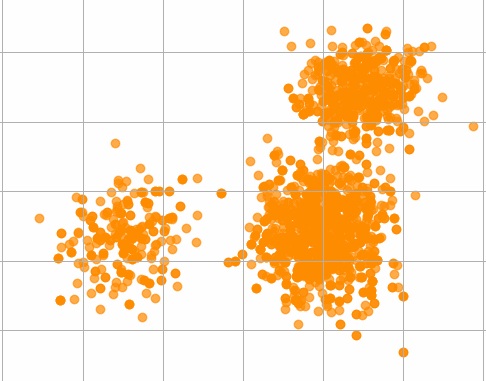}
    \caption{$t=0$}
    \label{fig:oakley:a}
  \end{subfigure}
  \hfill
  \begin{subfigure}[b]{0.25\linewidth}
    \centering
    \includegraphics[width=\linewidth]{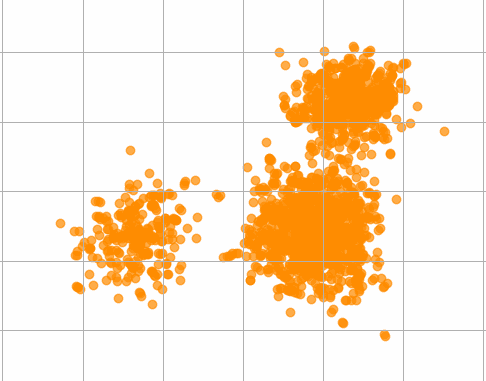}
    \caption{$t=0.01$}
    \label{fig:oakley:b}
  \end{subfigure}
  \hfill
  \begin{subfigure}[b]{0.25\linewidth}
    \centering
    \includegraphics[width=\linewidth]{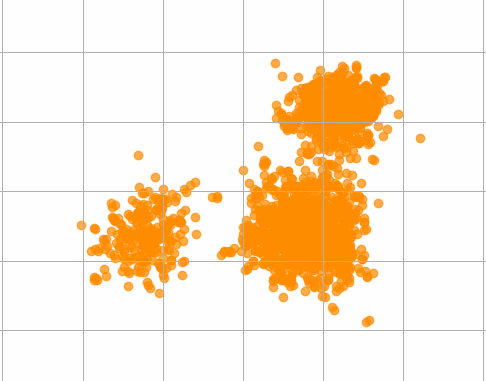}
    \caption{$t=0.02$}
    \label{fig:oakley:c}
  \end{subfigure}

  % ---- panel (d): full-width subfigure ----
  \vskip 1em
  \begin{subfigure}[b]{\linewidth}
    \centering
    \includegraphics[width=\linewidth]{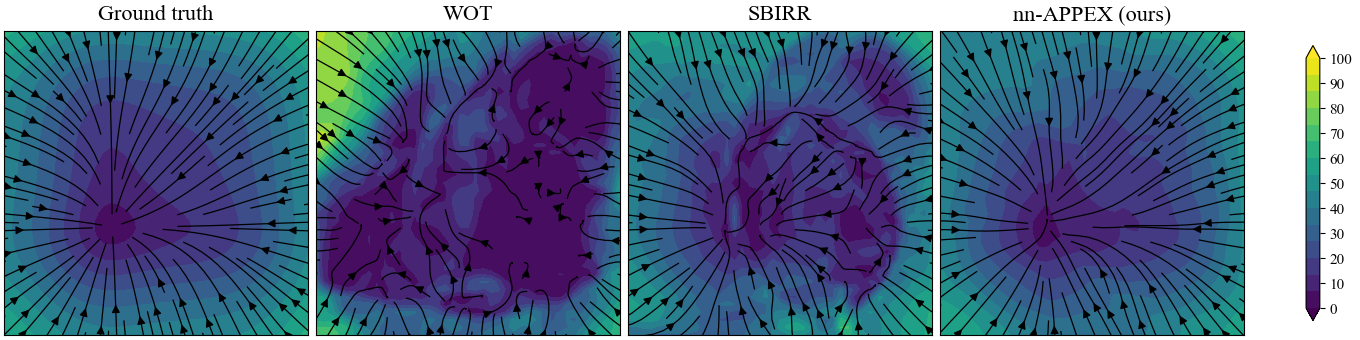}
    \caption{True and estimated landscapes}
    \label{fig:oakley:d}
  \end{subfigure}

  \caption{We simulate gradient-flow SDEs from a variety of potentials and provide inference methods with samples from three distinct marginals, initialized from a random Gaussian mixture model. Data for one seed is plotted for the Oakley–O'Hagan potential in (a)–(c), along with the true and estimated landscapes in (d).}
  \label{fig:oakley-gmms-landscapes}
\end{figure*}

\begin{figure*}[t]
    \centering

    % --- Top subfigure ---
    \begin{subfigure}{\textwidth}
        \centering
        \textbf{(a) Absolute drift error (normalized by true magnitude)}\\[0.3em]
        \includegraphics[width=\textwidth]{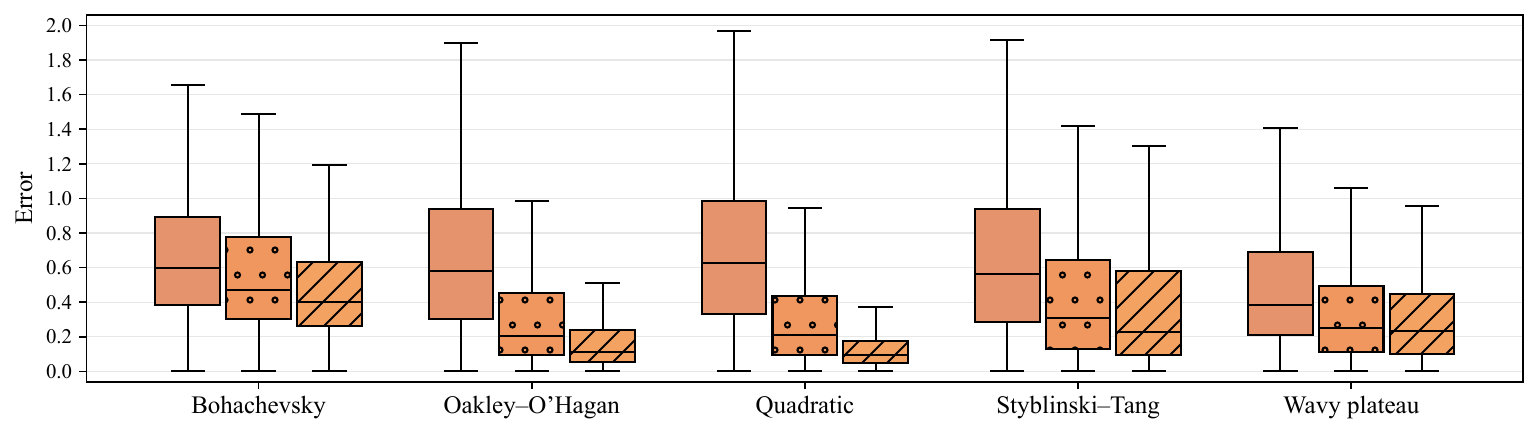}
        \label{fig:box_mae_gmm}
    \end{subfigure}

    \vspace{1em}

    % --- Bottom subfigure ---
    \begin{subfigure}{\textwidth}
        \centering
        \textbf{(b) Cosine similarity}\\[0.3em]
        \includegraphics[width=\textwidth]{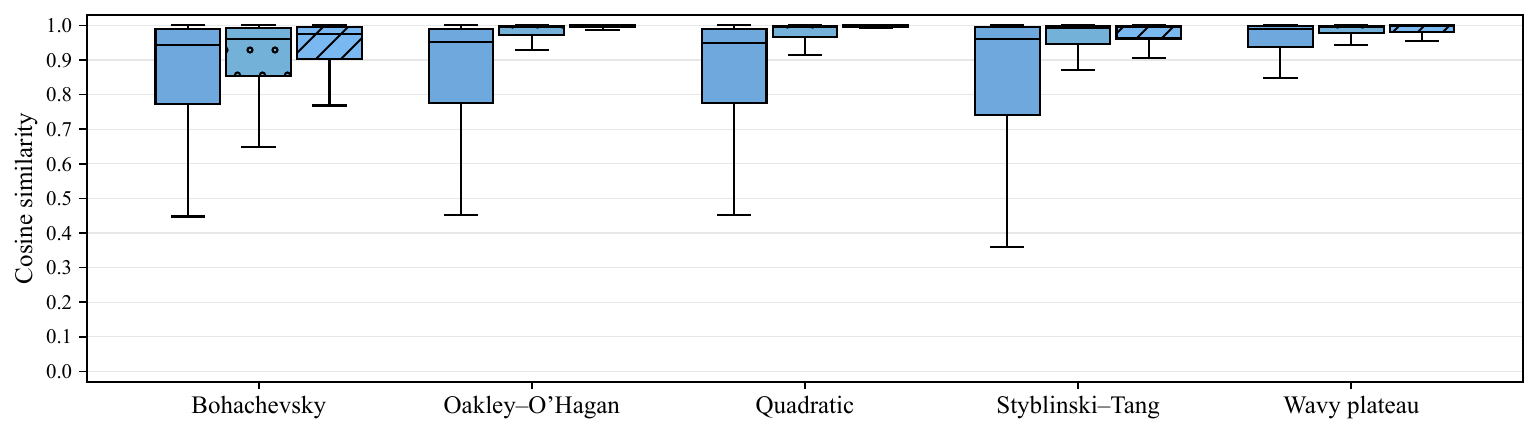}
        \label{fig:box_cos_gmm}
    \end{subfigure}

    \vspace{1em}

    % --- Legend centered below ---
    \begin{tikzpicture}
      \centering
      % sizes
      \def\w{1}   % rectangle width (cm)
      \def\h{0.4}   % rectangle height (cm)
      \def\dx{4.0}   % horizontal spacing (cm)
      \def\labgap{0.2} % label gap (cm)
      \def\hatchdist{4pt}
      \def\hatchwidth{0.3pt}
      \def\hatchangle{45}

      % WOT
      \begin{scope}[shift={(0,0)}]
        \draw[black] (0,0) rectangle (\w,\h);
        \node[anchor=west] at (\w+\labgap,\h/2) {\wot};
      \end{scope}

      % SBIRR
      \begin{scope}[shift={(\dx,0)}]
        \draw[black] (0,0) rectangle (\w,\h);
        \draw[black] (0.15,0.5*\h) circle (0.02);
        \draw[black] (0.5,0.5*\h) circle (0.02);
        \draw[black] (0.85,0.5*\h) circle (0.02);
        \node[anchor=west] at (\w+\labgap,\h/2) {\sbirr};
      \end{scope}

      % nn-APPEX
      \begin{scope}[shift={(2*\dx,0)}]
        \path[
          pattern = {Lines[angle=\hatchangle, distance=\hatchdist, line width=\hatchwidth]},
          pattern color = black
        ] (0,0) rectangle (\w,\h);
        \draw[black] (0,0) rectangle (\w,\h);
        \node[anchor=west] at (\w+\labgap,\h/2) {\nappex{} (ours)};
      \end{scope}
    \end{tikzpicture}
    \caption{The ability of different Schr\"odinger Bridge methods to infer the gradient-flow drift is evaluated across five different potentials using (a) normalized absolute error (lower is better) and (b) cosine similarity (higher is better). Methods are given samples from three distinct marginals, such that the initial distribution is a Gaussian mixture model with randomly initialized components. The box-and-whisker plots aggregated from $10$ seeds show that our method, \nappex{}, performs the best across all potentials.}
    \label{fig:boxplots_gmm}
\end{figure*}

The SB problem \eqref{eq: SB_traj_inf_k} is equivalent to entropic optimal transport with an adjusted cost, and can thus be solved with a variety of iterative proportional fitting algorithms, such as Sinkhorn's algorithm  \citep{peyre2019computational}. By default, we use the multi-marginal Sinkhorn algorithm \citep{marino2020optimal}. For MLE parameter estimation on the inferred paths, we use the Euler-Maruyama approximation, and present derivations in \cref{sec: MLE_appendix}. In particular, the optimal drift parameters $\hat{\theta}$ are independent of $\sigma^2$, since minimizing the negative log-likelihood with respect to $\theta$ is equivalent to minimizing
\begin{equation}
\label{eq:likelihood_fn_nn}
\textstyle \ell(\theta) = \scalebox{0.93}{$ \displaystyle \sum_{m=1}^{M}\sum_{i=0}^{N-2}\left\|x_{i+1}^{(m)}-x_{i}^{(m)}+\nabla \Psi_\theta(x_{i}^{(m)})\Delta t\right\|_2^2 $},
\end{equation}
where $M$ is the number of paths, $N$ is the number of time steps, and $\Delta t$ is the time step. We therefore estimate $-\nabla \Psi_{\hat{\theta}}$ by fitting a neural network whose parameters minimize $\ell(\theta)$, as done in \citet{shen2025multi}. The diffusion MLE, $\hat{\sigma}^2$, is then derived from the quadratic variation, conditioned on the estimated drift $\nabla \Psi_{\hat{\theta}}$,
% $\hat{\sigma}^2=\frac{1}{d\,N(T-1)\,\Delta t}\ell(\hat{\theta})$.
\begin{align}
\hat{\sigma}^2
&=
\frac{1}{d\,M\,(N-1)\,\Delta t}\ell(\hat{\theta}).
\label{eq:quad_var_diff_MLE}
\end{align}

\begin{algorithm}
\caption{\nappex{}}
\begin{algorithmic}[1]
\State \textbf{Input:} Observed marginals $\hat{p}_{t_i}$, $i = 0, \ldots, N-1$, number of iterations $K$, time step $\Delta t$
\State \textbf{Initialize:} $\nabla\hat{\Psi} \gets 0$,\; $\hat{\sigma}^2 \gets 1$
\For{$k = 1, \ldots, K$}
    \State $\{\mathcal{K}_i(x,y)\}_{i=0}^{N-2} \gets \exp\!\left(-\frac{\|y - x + \nabla\hat{\Psi}(x)\,\Delta t\|^2}{ 2\hat{\sigma}^2\Delta t}\right)$
\State $\{\pi_i\}_{i=0}^{N-2}  \gets \textsc{MultiSinkhorn}\!\left(\{\hat{p}_{t_i}\},\, \{\mathcal{K}_i\}\right)$
\vspace{0.5em}
\State $\{\hat{\tau}^{(s)}\}_{s=1}^S \gets \textsc{SamplePaths}\!\left(\{\pi_i\}\right)$
\vspace{0.5em}
    \State $\nabla\hat{\Psi} \gets \textsc{NNTrain}\!\left(\{\hat{\tau}^{(s)}\}\right)$
    \State $\hat{\sigma}^2 \gets \textsc{MLEDiffusivity}\!\left(\{\hat{\tau}^{(s)}\},\, \nabla\hat{\Psi}\right)$
\EndFor
\State \Return $\nabla\hat{\Psi},\; \hat{\sigma}^2$
\end{algorithmic}
\label{alg:nappex}
\end{algorithm}

By alternating between these three procedures, \nappex{} iteratively reduces the divergence between \emph{reconstructed paths}, $P \in \Pi\left(p(\cdot, t_i)_{i=0}^{N-1}\right)$, which obey the observed marginals, and \emph{paths from the estimated gradient-flow SDE} with law $Q$:
% In contrast, \nappex{} benefits from extra flexibility with adaptive diffusivity, while being able to estimate general gradient-flows. By construction, its estimates will iteratively approach the true parameters:
\begin{align}
\mathrm{KL}(P_{k+1}\|Q_{\shortminus\nabla\Psi_{k+1}, \sigma^2_{k+1}}) &\le \mathrm{KL}(P_{k+1}\|\,Q_{\shortminus\nabla\Psi_{k+1}, \sigma^2_{k}}) \notag \\
\le \mathrm{KL}(P_{k+1}\|Q_{\shortminus\nabla\Psi_{k}, \sigma^2_{k}}) &\le \mathrm{KL}(P_{k}\|Q_{\shortminus\nabla\Psi_{k}, \sigma^2_{k}}). 
\label{eq: appex_decrease_KL}
\end{align}

By decreasing the relative divergence, \nappex{} iteratively approaches the unique true solution, though it is unclear whether the algorithm can stagnate at a different set of parameters. We conjecture that convergence to the true solution holds given infinite data, and leave the proof for future work. 

To highlight distinctions with existing SB methods, we note relative differences in algorithmic design:
% \eqref{eq: SB_traj_inf_k}-\eqref{eq: diff_MLE_k}:
\begin{itemize}
    \item Waddington-OT (\wot) \citep{schiebinger2019optimal} fixes the reference $Q$ to be a Brownian motion with diffusivity $\sigma^2$ and stops  after \eqref{eq: drift_MLE_k}, i.e., fix $K=1$ in \cref{alg:nappex}.
    \item \sbirr{} \citep{shen2025multi,zhang2024joint} fixes $\sigma^2$ and iteratively performs steps \eqref{eq: SB_traj_inf_k} and \eqref{eq: drift_MLE_k}, i.e., exclude line 8 in \cref{alg:nappex}.
    \item \appex{} \citep{guan2024identifying} iteratively performs all three steps, but only considers linear drift, i.e., replace line 7 in \cref{alg:nappex} with a closed form MLE for linear SDEs.
\end{itemize}
Each method reconstructs paths by solving a Schr\"odinger Bridge problem. However, these paths will be incorrect unless the reference SDE matches the true SDE. While \nappex{} has the flexibility to re-estimate the gradient-flow drift as well as the diffusivity, previous methods cannot reliably determine a suitable reference SDE. \wot{}, and related methods \citep{lavenant2021towards, forrow2021lineageot, chizat2022trajectory}, fix a pure Brownian reference, which is incompatible with any nonzero potential. Although \sbirr{} adjusts its reference drift, if $\sigma^2$ is misspecified, then it would also misspecify the reference SDE, hindering its ability to learn the drift. While \appex{} adjusts drift and diffusion, it estimates a misspecified SDE, unless the ground-truth potential $\Psi$ happens to be quadratic.

% Recall that if the candidate set $\mathcal{S}$ is countable, e.g. as represented on a computer, then, by \cref{cor:three_points}, there can only be one compatible diffusivity $\sigma^2$ given $N \ge 3$ transient marginals.

% In particular, there is a unique minimizer for the objective
% our identifiability results show that there can only be one compatible diffusivity $\sigma^2$ given transient observations. 

% %
% Third, we observe that \cref{prop:stationary_non_iden} does not need the drift to be a gradient field and, at first sight, neither does \cref{thm:identifiability}. However, as we discuss in \cref{sec: non_iden_examples}, a Langevin SDE is not identifiable for arbitrary drift fields with rotational components. This assumption is indeed necessary in \citep[Theorem 2.1]{lavenant2021towards}, which relies on the theory of gradient flows \citep{ambrosio2007gradient}.

% The proof of the previous theorem uses the fact that under continuous observation, nonidentifiability requires each observed marginal $p(\cdot, t)$ to be stationary distribution of a residual Langevin SDE. 

\section{SIMULATED EXPERIMENTS}
\label{sec:experiments}
In this section, we perform extensive experiments on simulated data to evaluate \nappex{} against previous Schr\"odinger Bridge methods, \wot{} and \sbirr. The code repository is available on GitHub: \url{https://github.com/guanton/identifying-gradient-flow-sdes}.

% to systematically evaluate the importance of each aspect of the tri-level architecture \eqref{eq: SB_traj_inf_k}-\eqref{eq: diff_MLE_k}.% to investigate the importance of accurate diffusion estimation towards accurate drift estimation. 
\subsection{Experimental Setup}
We explain details about the data, method implementation, and performance metrics below. 
\paragraph{Data generation.}
To simulate gradient-flow SDEs, we fix the diffusivity $\sigma^2=0.2$ and consider five potentials that are commonly used as benchmarks in the literature \citep{terpin2024learning, persiianov2025learning}. Potentials were chosen to ensure the existence of a stationary distribution $p_{\mathrm{eq}}=\frac{1}{Z}\exp(-\frac{\Psi}{2\sigma^2})$. See \cref{sec:potentials} for the list of potentials.
% ; all of which admit a stationary distribution $p_{\mathrm{eq}}=\exp(-\frac{\Psi}{2\sigma^2})$.
For our main experiments, we initialize the first marginal as a random Gaussian mixture model (GMM), due to their universal approximation properties \citep{mclachlan2014number} and their applicability for clustering scRNA data into different cell types \citep{yu2021scgmai}. We uniformly sample the number of components between $1$ and $10$, their means within $[-3,3]^2$, and their variances between $0.5$ and $1.0$. We then mirror the ``three marginals'' identifiability setting from \cref{cor:three_points}, by sampling $N=2000$ points from the GMM initial distribution, and then forward simulating these points using the Euler-Maruyama scheme for an additional $2$ time steps of size $\Delta t = 0.01$. Experiments are repeated across $10$ seeds for each SDE to ensure replicability. Marginals for one seed are pictured for the SDE driven by the Oakley-O'Hagan potential in \cref{fig:oakley-gmms-landscapes}(a)-(c).
% We also report results for $p_0 \sim \textrm{Unif}[-4,4]^2$ and $p_0 \sim p_{\mathrm{eq}}$ in \cref{sec: additional_experiments_appendix} of the Supplement.

\paragraph{Methods.} 
To perform a systematic ablation that isolates the effects of iterative reference refinement and adaptive diffusion on inference, all methods are given the same data and initialization, and any shared subprocedure is implemented identically across methods. To emulate the realistic scenario where the practitioner can only estimate the true diffusivity $\sigma^2$ up to an order of magnitude \citep{guan2024identifying}, we sample the diffusivity prior as $\sigma^2_{\text{prior}} \sim \sigma^2 \times 10^{\textrm{Unif}[-1,1]}$, for each of the $10$ seeds. This setting could be considered favourable towards methods like \wot{} and \sbirr{}, since real-world diffusion estimates may in fact be off by two orders of magnitude \citep{forrow2024consistent}. For all methods, the SB step \eqref{eq: SB_traj_inf_k} is solved using up to $200$ iterations of a multi-marginal iterative proportional fitting procedure (IPFP) \citep[Section 4.3]{marino2020optimal}, which rescales the per-time slice scalings until the marginal constraints are met up to an $L^\infty$ error of $10^{-5}$. $2000$ trajectories are then sampled from the inferred law on paths. Each method performs the drift MLE step \eqref{eq: drift_MLE_k} by fitting a $2$-layer (128 neurons per layer) multi-layer perceptron with \texttt{SiLU} activation, trained to minimize \eqref{eq:likelihood_fn_nn} via the Adam optimizer (epochs $=500$ and learning\_rate $=3 \times 10^{-3}$) \citep{kingma2014adam}. Only \sbirr{} and \nappex{} are iterative, and we use $30$ iterations for each method.

% used for iterative refinement in \sbirr{} and \nappex is implemented as a multi-layer perceptron 

% we compare the performance of methods built on the previous partial identifiability theory, given the realistic setting of an incorrect diffusivity prior 
\vspace{-1em}
\paragraph{Metrics.}
We assess \wot, \sbirr, and \nappex, by evaluating both the scale and shape of their resulting landscapes $\nabla \Psi_{\hat{\theta}}$ with respect to the true landscape $\nabla \Psi$. We use the following metrics, evaluated on $2601$ points ($51 \times 51$ point grid) within $[-4,4]^2$: 
% via normalized absolute error and cosine similarity: 
\begin{itemize}[leftmargin=*, nosep]
\item \emph{Absolute error (normalized)}: $\displaystyle \frac{|\nabla \Psi_{\hat{\theta}}(x)-\nabla \Psi(x)|}{\|\nabla \Psi(x)\|}$
\item \emph{Cosine similarity}: $\displaystyle \frac{\langle \nabla\Psi_{\hat{\theta}}(x), \hspace{0.1cm}\nabla \Psi(x) \rangle }{\|\nabla\Psi_{\hat{\theta}}(x)\| \|\nabla \Psi(x)\|}$
\end{itemize}

% While \wot{} and \sbirr{} do not estimate diffusion from marginals, we also compare the drift and diffusion estimates of \nappex{} and \jkonetstar{} (a non-SB-based method) in the Supplement.

% \begin{itemize}[leftmargin=*]
%     \item \emph{Mean absolute error for estimated SDE parameters}:  We compute the mean absolute errors (MAE) for the drift field. We compute the average MAE across $M$ evaluated points in the grid $[-4,4]^2$ between the true and estimated drift vector fields, and normalize by the average magnitude of the true vector field.
%     \item \emph{Cosine similarity}: We use , averaged across $M$ points in the grid $[-4,4]^2$ to measure the performance of drift direction estimation.
% \end{itemize}
\subsection{Results}
The results are aggregated over $10$ seeds for each gradient-flow SDE, and shown in \cref{fig:boxplots_gmm}. We see that \nappex{} yields the best and most robust estimates for the drift landscapes. On each of the five SDEs, its estimates have the lowest absolute error, the highest cosine similarity, and the lowest variance. While \sbirr{} significantly outperforms \wot, due to iterative drift refinement, its inability to update its diffusion prior prevents it from learning the drift field to the same fidelity as \nappex{} (see also \cref{tab:gradient_magnitude} and \cref{tab:performance_comparison_mean_median} in the Appendix). In particular, \wot{} and \sbirr{} regularly orient flow lines incorrectly in low potential zones, due to misspecified diffusivity (see \cref{fig:oakley-gmms-landscapes} and \cref{tab:gradient_magnitude}).

We also replicate the experiment such that the initial distribution is either uniform, $p_0 \sim \textrm{Unif}[-4,4]^2$, as done in previous work \citep{terpin2024learning}, or given by the stationary Gibbs distribution $p_0 \sim p_{\mathrm{eq}}$. The results are summarized in \cref{fig:boxplots_unif} and \cref{fig:boxplots_gibbs}. For the uniform initialization, \nappex{} continues to perform the best on all SDEs, and all methods perform relatively better, due to higher observability in the evaluation region $[-4,4]^2$. For the Gibbs initialization, all methods fail at inference, which is consistent with identifiability theory (\cref{prop:stationary_non_iden}). Finally, we evaluate \nappex{} against the state-of-the-art variational method \jkonetstar{} \citep{terpin2024learning} given GMM initializations. Results are plotted in \cref{fig:boxplots_jko_appex} and show that \nappex{} achieves more accurate drift and diffusion estimates.

\section{BIOLOGICAL EXPERIMENTS}
We now test \nappex's ability to perform trajectory inference on real biological data. As done in \citet{shen2025multi}, we use single-cell data from human embryonic stem cells (hESC) \citep{chu2016single}.

\subsection{Experimental Setup}

For consistency, we use the same data preprocessing, method implementation, and performance evaluation as \citet{shen2025multi}. Our code builds on the \sbirr{} public repository, such that the only modifications were to randomize the initial diffusivity, and to add the diffusion update step to implement \nappex.

\paragraph{Dataset.} The hESC dataset \citep{chu2016single} comprises $5$ time marginals ($0$h, $12$h, $24$h, $36$h, $72$h), and the observed number of cells per marginal are: $92$, $102$, $66$, $172$, $138$. $5$ PCA components are used.

\paragraph{Method implementation.} For this experiment, the SB solver for each method is implemented with the forward-backward iterative proportional maximum likelihood algorithm \citep{vargas2021solving}. \sbirr{} and \nappex{} both use the same gradient field MLE for re-estimating the reference drift (see 
\citet[Section D.6.2]{shen2025multi} for details), and we add the closed form diffusivity MLE \eqref{eq: diff_MLE_k} for re-estimating the reference diffusion for \nappex. We also implement a time-varying version of \nappex, which computes \eqref{eq: diff_MLE_k} between each pair of marginals to estimate time-varying diffusivity. All iterative methods are run for $10$ iterations. 

\paragraph{Evaluation.} We train each method on the first (0h), third (24h), and fifth (72h) marginals, and then use the learned dynamics to predict the second (12h) and fourth (36h) marginals. We then evaluate performance using the earth mover's distance between the estimated and true holdout marginals. We average results over $10$ random seeds, and for each seed, we sample a random initial diffusivity within an order of magnitude of $\sigma^2 = 0.1$, the default considered in \citet{shen2025multi}.

\subsection{Results} 
We report results in Table \ref{tab:emd_hESC}. We observe that \sbirr’s iterative drift refinement enables better performance compared to the simplest method \wot. Adding the diffusion update then provides another small but noticeable improvement, such that \nappex{} (with time-varying diffusivity) achieves the best overall performance on both holdout marginals. 

This experiment enhances our point that diffusion should be learned from the data, rather than assumed beforehand. Indeed, for real data, the diffusion is unknown, and it is often not obvious how it should be specified. Methods that refine both drift and diffusion also predict unseen marginals with higher accuracy, which suggests that they learn a better model. 
% We further believe that overall performance of all models would be 

\begin{table}[ht]
\caption{The average earth mover's distance ($\pm$ standard error) across each method and each hold out time on the hESC dataset. The best results are displayed in bold and the second best results are underlined.}
\label{tab:emd_hESC}
\centering
\begin{tabular}{llc}
\toprule
\textbf{Method} & \textbf{Time} & \textbf{Avg. EMD} \\ \midrule
\multirow{2}{*}{\textbf{\wot}} & $t_2$ & 0.803 $\pm$ 0.039 \\
 & $t_4$ & 1.492 $\pm$ 0.081 \\ \midrule
\multirow{2}{*}{\textbf{\sbirr}} & $t_2$ & 0.694 $\pm$ 0.021 \\
 & $t_4$ & 1.473 $\pm$ 0.033 \\ \midrule
\multirow{2}{*}{\textbf{\nappex}} & $t_2$ & \underline{0.683} $\pm$ 0.021 \\
 & $t_4$ & \underline{1.470} $\pm$ 0.040 \\ \midrule
\multirow{2}{*}{\makecell[l]{\textbf{\nappex} \\ {(time-varying)}}} & $t_2$ & \textbf{0.678} $\pm$ 0.022 \\
 & $t_4$ & \textbf{1.454} $\pm$ 0.033 \\ \bottomrule
\end{tabular}
\end{table}

\section{RELATED WORKS} 
Our work contributes novel theory and methodology for identifying the true SDE from observed marginals. We review relevant literature below.
% Recent work has provided the first rigorous conditions for identifying an SDE solely from observational data. In particular, \citet{wang2024generator} proved necessary and sufficient conditions for identifying linear additive noise SDEs, given observed trajectories. Following this, \citet{guan2024identifying} proved necessary and sufficient conditions for identifying the same SDE, given observed marginals.

% While recent work has established conditions for identifying the drift and diffusion parameters of SDEs from observational data, these results only consider SDEs with linear drift. In the setting of multiple observed trajectories, \citet{wang2024generator} proved necessary and sufficient identifiability conditions when the noise is additive, and sufficient identifiability conditions when the noise is multiplicative. In the more general setting of observed marginals, \citet{guan2024identifying} proved necessary and sufficient identifiability conditions when the noise is additive.
\paragraph{Identifiability theory for population dynamics.} 
% SDE identifiability is an active area of research. 
To the best of our knowledge, conditions for jointly identifying an SDE's drift and diffusion from marginals have only been proven for linear additive noise SDEs \citep{guan2024identifying}, which cannot capture the multistable dynamics of cell differentiation. 
% In contrast, Langevin SDEs model general potential-driven dynamics.
% Despite the prevalence of the Langevin SDE, and a flurry of machine learning inference methods developed in the last decade to infer it from marginals \citep{hashimoto2016learning, neklyudov2023action, lavenant2021towards, yeo2021generative, zhang2021optimal, chizat2022trajectory,  bunne2022proximal, shen2025multi,zhang2024joint, guan2024identifying, terpin2024learning, persiianov2025learning}, 
% % no existing theory has established conditions under which it can be fully identified from its marginals. 
% the most comprehensive results in the literature only prove that the drift $-\nabla \Psi$ can be identified from $\bigl(p(\cdot, t)\bigr)_{t \in [0,T]}$ if the diffusivity $\sigma^2$ is known:
% \begin{mdframed}[hidealllines=true,backgroundcolor=orange!5]
% \begin{equation*}
% \text{observe }\hspace{0.1cm} \bigl(p(\cdot, t)\bigr)_{t \in [0,T]} \quad \text{and} \quad \sigma^2 \hspace{0.1cm} \text{ is known}\quad \implies \quad {-\nabla \Psi} \hspace{0.1cm} \text{ is identifiable}.
% \label{eq: previous_identifiability}
% \end{equation*}
% \end{mdframed}
For the identifiability of gradient-flow SDEs, the partial identifiability result \eqref{eq:previous_identifiability} was first proven by \citet{hashimoto2016learning}. \citet{lavenant2021towards} extended the result for time-inhomogeneous drift $-\nabla \Psi(x,t)$, and showed that a single sample per marginal suffices, if the measurement times are dense. Finally, \citet{neklyudov2023action} proved that the result also holds if the known diffusivity $\sigma(t)^2$ is time-inhomogeneous. 

\paragraph{Schr\"odinger Bridge based inference methods for population dynamics.} 
% There exist dozens of methods for SDE inference from temporal marginals.
A popular approach for marginals-based inference is to reconstruct trajectories with entropic optimal transport, which is equivalent to solving a Schr\"odinger Bridge (SB) problem \citep[Prop. 4.2]{peyre2019computational}. The first such method was Waddington-OT (\wot) \citep{schiebinger2019optimal}, which performs trajectory inference by solving a single SB problem with respect to a Brownian motion reference with a prescribed diffusivity $\sigma^2$. \wot{} spawned many variations, which incorporate additional information, like cell lineages \citep{forrow2021lineageot,ventre2023trajectory}, or improve robustness given limited samples \citep{zhang2021optimal,lavenant2021towards,chizat2022trajectory}. In the last year, SB methods were improved with the introduction of iterative reference refinement, i.e. \sbirr{} \citep{shen2025multi, zhang2024joint}. Rather than fixing Brownian motion as a reference, \sbirr{} fixes the diffusivity $\sigma^2$ and iteratively re-estimates the drift of the reference process. In contrast to \sbirr, the method that we introduce in this work, \nappex, additionally re-estimates the diffusivity at each iteration. \nappex{} builds upon the three-stage iterative scheme of \appex{} \citep{guan2024identifying}.
% \rv{While \appex{} was specifically developed for SDEs with linear drift and additive diffusion, \nappex{} can be used for SDEs with gradient-flow drift and isotropic diffusion. This added flexibility for modeling nonlinear drift comes from the fact that we estimate drift with a neural network rather than a pre-parameterized likelihood based on the assumption of linearity.} 

% We explain the key details of our method and how its methodology compares against existing SB methods in \cref{sec:method}.

% \nappex{} can also be viewed as a tri-level extension to bi-level iterative refinement schemes, like \sbirr \citep{shen2025multi}, which keep the diffusivity constant across iterations. 
\paragraph{Other inference methods for population dynamics.}
In addition to SB methods, variational energy-based methods \citep{hashimoto2016learning,bunne2022proximal,neklyudov2023action, terpin2024learning, persiianov2025learning} have been developed for SDE inference from population dynamics. These methods leverage the fact that the marginals of potential-driven SDEs minimize a corresponding energy, and use discrete numerical schemes to approximate the underlying parameters. Of these, only \jkonetstar{} \citep{terpin2024learning} jointly estimates drift and diffusion. We note that \jkonetstar{} and \nappex{} consider distinct optimization approaches (the former leverages the JKO scheme \citep{jordan1998variational} with a joint variational step, whereas \nappex{} performs three-stage SB refinement) and we compare their performance in \cref{sec: additional_experiments_appendix}. We also note that \citet{terpin2024learning} did not study identifiability theory, and therefore lacked guarantees for principled inference. Recent work presents another framework for joint inference via maximum mean discrepancy minimization \citep{berlinghieri2025oh}.

 % While \citep{terpin2024learning} tackles joint parameter estimation for the Langevin SDE, their work did not study identifiability theory, and therefore lacked guarantees for principled inference. We also note that interesting concurrent work presents a third framework for joint inference, via maximum mean discrepancy (MMD) \citep{berlinghieri2025oh}.

\section{CONCLUSIONS}
\label{sec:conclusions}

% Inferring gradient-flow SDEs from their temporal marginals is an important problem in many areas of applied research, including single-cell biology, ground water hydrology, and machine learning. In the last decade, inference methods have been developed based on existing conditions for identifiability, which required a known diffusion parameter to identify the drift. 
This work provides a call to action for researchers in ML and single-cell biology to move beyond the current paradigm of assuming known diffusion for marginals-based SDE inference, as we show that this is not only unnecessary, but also problematic in practice. Our theoretical contributions resolve a longstanding identifiability problem by proving that the gradient-flow drift and the diffusivity are jointly identifiable if and only if we observe transient marginals. This is significant for practical applications, since transience is a common and verifiable condition. We further expand the practicality of our result by proving that only three distinct marginals are needed for identifiability. To translate this theory into practice, we introduced the first Schr\"odinger Bridge-based method capable of inferring arbitrary gradient-flow drift and diffusivity. Extensive experiments demonstrate that our method, \nappex, outperforms previous SB methods, in the realistic scenario where diffusivity is unknown.

% In sum, our work provides novel theory and methodology for the accurate recovery of trajectories, gene regulatory networks, and post-intervention distributions from observed marginals.

\paragraph{Limitations and future work.}
While our work contributes novel theory and methodology for SDE inference from marginals, we note several limitations, which point to important directions for future work. First, the gradient-flow SDE model \eqref{eq:overdamped_langevin_SDE} is prevalent in many fields, but it does not capture some important dynamics. In single-cell dynamics, one may expect rotational dynamics from non-conservative drift, due to genetic feedback loops. As another example, in hydrology, material heterogeneity requires modified models for drift and diffusion. Second, our theory was proven given exact observation of the marginals $p(\cdot, t)$. Rigorous asymptotic theory quantifying identifiability with finite samples and observational noise would be useful to the community. Finally, there is still a gap in showing optimality of joint inference methods. While experiments show positive results, proving that \nappex{} converges to the true parameters under basic conditions would mark a major advance.
% It has been conjectured that three transient marginals $p_{t_0} \neq p(\cdot, t_1) \neq p(\cdot, t_2)$ suffice for identifiability, given that the diffusivity is known \citep{hashimoto2016learning}. 
% Important questions also remain with respect to practical identifiability. Although our experiments in~\cref{sec: fisher_experiments} provide heuristics for optimal measurement strategies, we note that a rigorous analysis, which quantifies the difference between the Fisher information from observed trajectories versus observed marginals, remains an important direction for future work. 

\paragraph{Acknowledgments.}

We thank NSERC for their support for Vincent Guan and Joseph Janssen. Nicolas Lanzetti was supported by the NCCR Automation, a National
Centre of Competence in Research, funded by the Swiss National Science
Foundation (grant number 51NF40\_225155). Elina Robeva was supported by a Canada CIFAR AI Chair and an NSERC Discovery Grant (DGECR-2020-00338). The authors would also like to thank United Therapeutics for supporting this research.

\bibliographystyle{plainnat}
\bibliography{references}

\section*{Checklist}

\begin{enumerate}

  \item For all models and algorithms presented, check if you include:
  \begin{enumerate}
    \item A clear description of the mathematical setting, assumptions, algorithm, and/or model. [\textbf{Yes}/No/Not Applicable]\\
    \textit{The data-generating model is defined in \eqref{eq:overdamped_langevin_SDE} and precise mathematical details and blanket assumptions are given in \cref{sec:math_background}. Our inference algorithm is defined in \cref{sec:method} and further details are given in \cref{sec: supplement_sb_refinement}}
    \item An analysis of the properties and complexity (time, space, sample size) of any algorithm. [\textbf{Yes}/No/Not Applicable]\\
    \textit{The mathematical properties of our algorithm are discussed in \cref{sec:method}, with related derivations in \cref{sec: supplement_sb_refinement} and \cref{sec: MLE_appendix}. Relevant numerical details are provided in \cref{sec:experiments}, and runtimes are reported in \cref{tab:appex_timing_overall}.}
    \item (Optional) Anonymized source code, with specification of all dependencies, including external libraries. [\textbf{Yes}/No/Not Applicable]
  \end{enumerate}

  \item For any theoretical claim, check if you include:
  \begin{enumerate}
    \item Statements of the full set of assumptions of all theoretical results. [\textbf{Yes}/No/Not Applicable]\\
    \textit{All assumptions, outside the blanket assumptions (two textbook assumptions required for SDE existence and uniqueness), are contained within the statements of the theoretical result.}
    \item Complete proofs of all theoretical results. [\textbf{Yes}/No/Not Applicable]\\
    \textit{Given the importance of our theory and the relative brevity of the proofs, we include full proofs in \cref{sec:iden_theory}, under \cref{thm:identifiability} and \cref{cor:three_points} respectively. Technical details related to our inference algorithm are derived in \cref{sec: supplement_sb_refinement} and \cref{sec: MLE_appendix}.}
    \item Clear explanations of any assumptions. [\textbf{Yes}/No/Not Applicable]\\
    \textit{We discuss assumptions, their relevance for practical inference, and comparison with assumptions from previous literature throughout the paper. For example, see the discussion after the proof of \cref{thm:identifiability}. }
  \end{enumerate}

  \item For all figures and tables that present empirical results, check if you include:
  \begin{enumerate}
    \item The code, data, and instructions needed to reproduce the main experimental results (either in the supplemental material or as a URL). [\textbf{Yes}/No/Not Applicable]\\
    \textit{A zip file of the code repository with scripts for running the experiments, and a detailed README file are included in supplemental material.}
    \item All the training details (e.g., data splits, hyperparameters, how they were chosen). [\textbf{Yes}/No/Not Applicable]
    \textit{We report implementation details in \cref{sec:experiments} and emphasize that we use exact implementations for any shared procedures used by baselines.}
    \item A clear definition of the specific measure or statistics and error bars (e.g., with respect to the random seed after running experiments multiple times). [\textbf{Yes}/No/Not Applicable]\\
    \textit{We specify our metrics in \cref{sec:experiments} and state that  we report results over $10$ random seeds. }
    \item A description of the computing infrastructure used. (e.g., type of GPUs, internal cluster, or cloud provider). [\textbf{Yes}/No/Not Applicable]\\
    \textit{Experiments were run locally on a laptop. Details are given at the start of \cref{sec: additional_experiments_appendix}. }
  \end{enumerate}

  \item If you are using existing assets (e.g., code, data, models) or curating/releasing new assets, check if you include:
  \begin{enumerate}
    \item Citations of the creator If your work uses existing assets. [\textbf{Yes}/No/Not Applicable]\\
    \textit{We declare dependencies on two previous repositories in \cref{sec: additional_experiments_appendix} and cite the original works.}
    \item The license information of the assets, if applicable. [\textbf{Yes}/No/Not Applicable]\\
    \textit{We report the license in \cref{sec: additional_experiments_appendix} and also in our source code.}
    \item New assets either in the supplemental material or as a URL, if applicable. [\textbf{Yes}/No/Not Applicable]\\
    \textit{We provide the full code repository as supplemental material.}
    \item Information about consent from data providers/curators. [Yes/No/\textbf{Not Applicable}]
    \item Discussion of sensible content if applicable, e.g., personally identifiable information or offensive content. [Yes/No/\textbf{Not Applicable}]
  \end{enumerate}

  \item If you used crowdsourcing or conducted research with human subjects, check if you include:
  \begin{enumerate}
    \item The full text of instructions given to participants and screenshots. [Yes/No/\textbf{Not Applicable}]
    \item Descriptions of potential participant risks, with links to Institutional Review Board (IRB) approvals if applicable. [Yes/No/\textbf{Not Applicable}]
    \item The estimated hourly wage paid to participants and the total amount spent on participant compensation. [Yes/No/\textbf{Not Applicable}]
  \end{enumerate}

\end{enumerate}

\clearpage
\onecolumn
\appendix

\section{THEORETICAL BACKGROUND}
\label{sec: basic_langevin_properties}

\subsection{The Distributional Definition of the Fokker-Planck Equation}
\label{subsec:weak_FP}

The Fokker-Planck equation defines the evolution of an SDE's marginals. For an overdamped Langevin SDE \eqref{eq:overdamped_langevin_SDE}, its Fokker-Planck equation is given by
\begin{align}
    \frac{\partial p(x,t)}{\partial t} = \mathcal{L}^*_{\shortminus \nabla \Psi, \sigma^2}(p(\cdot,t))(x),
    \label{eq:FP_lang_appendix}
\end{align}
where the Fokker-Planck operator $\mathcal{L}^*_{\shortminus \nabla \Psi, \sigma^2}$ is defined as
\begin{align*}
(\mathcal{L}^*_{\shortminus \nabla \Psi, \sigma^2}p)(x)= \nabla \cdot(p(x) \nabla \Psi(x)) + \frac{\sigma^2}{2}\Delta p(x).
\end{align*}
However, the equation \eqref{eq:FP_lang_appendix} cannot be interpreted as a strong pointwise equality unless $p$ is a twice differentiable probability density and $\Psi \in C^2(\R^d)$, which is a stronger condition than Lipschitz continuity of $\nabla \Psi$. To show that the evolution of marginals remains well-defined when we consider a general probability measure $\mu$ with finite second moments, we consider the weak distributional formulation of the Fokker-Planck equation \eqref{eq:FP_lang_appendix} through its adjoint operator,
\begin{align*}
(\mathcal{L}_{\shortminus \nabla \Psi, \sigma^2}f)(x)=-\nabla \Psi(x) \cdot \nabla f(x)  + \frac{\sigma^2}{2}\Delta f(x).
\end{align*} 
% By the adjoint relationship, we have
% \begin{align*}
% \int_{\R^d} \mathcal{L}_{\shortminus \nabla \Psi, \sigma^2}f(x)g(x) \, \mathrm{d}x = \langle \mathcal{L}_{\shortminus \nabla \Psi, \sigma^2}f, g \rangle = \langle f, \mathcal{L}^*_{\shortminus \nabla \Psi, \sigma^2}g \rangle = \int_{\R^d} f(x)\mathcal{L}^*_{\shortminus \nabla \Psi, \sigma^2}g(x) \, \mathrm{d}x ,
% \end{align*}
% for all smooth and compactly supported test functions $f, g \in C^\infty_c(\R^d)$.

Consider now a family of Borel locally finite measures on $\R^d \times (0,T)$, which we denote by $(\mu_t)_{t \in (0,T)}$.
To define weak solutions, we follow~\citep[Definition 6.1.1, Proposition 6.1.2(ii)]{bogachev2022fokker} and say that $(\mu_t)_{t \in (0,T)}$ solves the Fokker-Planck equation for the initial condition 
 $\mu|_{t=0}=\nu$ in the weak sense
 % \footnote{Also  see~\citep{stroock2008partial}.} 
 if for all test functions $\varphi \in C_0^\infty(\R^d)$, we have that
\begin{align}\label{eq: weak_FP}
    \int_{\R^d}\varphi(x)\mathrm{d}\mu_t(x)
    -
    \int_{\R^d}\varphi(x)\mathrm{d}\nu(x)
    =
    \lim_{\tau\to 0+}\int_{\tau}^t\int_{\R^d}\mathcal{L}_{\shortminus \nabla \Psi, \sigma^2}\varphi(x)\mathrm{d}\mu_s(x)\mathrm{d}s
\end{align}
for almost all $t\in[0,T]$. This can be denoted using the compact notation 
\begin{align}
    \partial_t \mu = \mathcal{L}^*_{\shortminus\nabla \Psi, \sigma^2}\mu.
\end{align}
In this setting, the existence and uniqueness of a weak solution follows from~\citep[Theorem 9.4.8, Example 9.4.7]{bogachev2022fokker}, which applies because $\sigma^2$ is constant and we assumed that $\nabla \Psi$ is Lipschitz and obeys the growth condition, $\|\nabla\Psi(x)\|\le K(1+\|x\|)$. This result is also given in \citet[Theorem 1.1.9]{stroock2008partial}, which additionally shows that the family $(\mu_t)_{t \in (0,T)}$ is continuous under the topology of weak convergence. Thus, the evolution of marginals can be identified with a continuous transition probability function, even under the weak formulation.

\subsection{Additional Proofs}
\label{sec:additional_proofs}

The following results are used to complete the proof of \cref{cor:three_points}. Since the observed marginals in \cref{cor:three_points} are sampled after some time $T_1 > 0$, we note that it suffices to consider times $t \ge T_1 > 0$. Also recall that for this result, we assume that $\Psi \in C^\infty(\R^d)$. This facilitates the proof, by ensuring that any marginal $p(\cdot,t)$ of a candidate gradient-flow SDE exhibits nice regularity and decay. Indeed, by elliptic regularity theory, the Fokker-Planck operator $\mathcal{L}_{\shortminus \nabla \Psi, \sigma^2}$ smooths marginals within any finite time.

\begin{lemma}[Finite-time smoothing of marginals] Suppose that the potential of the gradient-flow SDE \eqref{eq:overdamped_langevin_SDE} is smooth, i.e. $\Psi \in C^\infty(\R^d)$. Then, for any positive time $t>0$, it follows that the marginal $p(\cdot, t)$ admits a twice differentiable density, which obeys the Gaussian decay estimates
\begin{align*}
    p(x,t) &\le Kt^{\frac{-d+2}{2}} \exp\bigl(-\frac{\delta}{2t}\|x\|^2 \bigr)\\
    \|\nabla p(x, t)\|  &\le Kt^{\frac{-d+2}{2}} \exp\bigl(-\frac{\delta}{2t}\|x\|^2 \bigr)
\end{align*}
for some $K, \delta > 0$.
\label{lemma: PDE_regularity_finite_time}
\end{lemma}
\begin{proof}
This result is given by \citet[Theorem 4.1]{pavliotis2014stochastic}. However, we note that in the setup of this theorem, the initial distribution $p(\cdot, 0)$ is assumed to have a density. Thus, we first apply the density estimates from \citet[Theorem 1.2]{menozzi2021density} to show that for any positive time $t/2>0$, the marginal $p(\cdot,{t/2})$ admits a density, even if $p(\cdot,0)$ is only a probability measure $\mu_0$. Indeed, by \citet[Theorem 1.2]{menozzi2021density}, the SDE admits a  transition density function, $p(x,t|y,0)$. Therefore, $p(\cdot, t) = \mu_0 \ast p(\cdot,t \mid \cdot,0)$ admits a density, since it is the convolution of a probability measure with a probability density \citep[Theorem 2.1.16]{durrett2019probability}.
% First, since we have the blanket assumption that  $\nabla \Psi$ satisfies the linear growth condition, and the diffusion is nondegenerate,  
We may thus apply \citet[Theorem 4.1]{pavliotis2014stochastic} after setting the initial density as the density of $p(\cdot, t/2)$. By combining these two PDE regularity theorems, we obtain the desired regularity and decay properties.
\end{proof}

\begin{lemma}[Equivalence of marginals and of flows at accumulation point] As defined in \cref{cor:three_points}, let $\mathcal{I}_i=\{t \in [T_i, T_{i+1}] \mid p(\cdot,t) = q(\cdot,t)\}$, and let $\{t_n\}_{n \ge 0}$ be a sequence of times in $\mathcal{I}_i$, which converges to $t_i^*$. Then, $p(\cdot, t_i^*)=q(\cdot, t_i^*)$ and $\frac{\partial}{\partial_t}p(\cdot, t_i^*) = \frac{\partial}{\partial_t} q(\cdot, t_i^*)$.
\label{lemma:flow_marginals_match}
\end{lemma}
\begin{proof}
First, we define the map $f(t):t \in [T_i, T_{i+1}]\to p(\cdot, t)-q(\cdot,t)$, where we recall that $T_i>0$. Then, by \cref{lemma: PDE_regularity_finite_time}, for any $t \ge T_i$, $p(\cdot, t)$ and $q(\cdot,t)$ each admit a smooth density. Since we also have $\Psi \in C^\infty(\R^d)$, it follows that the Fokker-Planck equation \eqref{eq:FP_eq_langevin} is defined pointwise, which implies the existence of the time derivatives $\frac{\partial p(x,t)}{\partial t}$ and $\frac{\partial q(x,t)}{\partial t}$ \citep[Theorem 4.1]{pavliotis2014stochastic}. Thus, $f(t)=p(\cdot,t)-q(\cdot,t)$ is also a strongly differentiable map in $[T_i,T_{i+1}]$. From this, we deduce that $\mathcal{I}_i$ is closed, since $\mathcal{I}_i=f^{-1}(\{0\})$, which is the preimage of a closed set under a continuous map. 
% We note that even without differentiability, this follows for the weak formulation, since we have continuity under the topology corresponding to weak convergence \citep[Theorem 1.1.9]{stroock2008partial}. 
Since $t_i^*$ is the limit point of a sequence $\{t_n\}_{n \in \N}$ in $\mathcal{I}_i$, it follows that $t_i^*$ is also in the coincidence set $\mathcal{I}_i$, which proves that $p(\cdot, t_i^*) = q(\cdot, t_i^*)$. In fact, since they each admit densities, this is equivalent to the pointwise equality, $p(x, t_i^*) = q(x, t_i^*)$ $\forall x \in \R^d$. Then, to prove that $\frac{\partial}{\partial_t} p(\cdot, t_i^*) = \frac{\partial}{\partial_t} q(\cdot, t_i^*)$, we recall that $p(\cdot,t)$ and $q(\cdot,t)$ are strongly differentiable in time, such that, for all $x \in \R^d$,
\begin{align*}
    \frac{\partial}{\partial t} p(x,t_i^*) &= \lim\limits_{t_n \to t_i^*} \frac{p(x,t_n)-p(x,t_i^*)}{t_n-t_i^*}\\
    \frac{\partial}{\partial t} q(x,t_i^*) &= \lim\limits_{t_n \to t_i^*} \frac{q(x,t_n)-q(x,t_i^*)}{t_n-t_i^*}.
\end{align*}
Then, by convergence of the sequence $t_n \to t_i^*$, the fact that $p(x,t_n)=q(x,t_n)$ (by construction of the sequence) and $p(x,t_i^*)=q(x,t_i^*)$, it follows that $\frac{\partial}{\partial t} p(x,t_i^*) = \frac{\partial}{\partial t} q(x,t_i^*)$ for all $x \in \R^d$.
\end{proof}

\begin{prop}[H-Theorem: marginals of gradient-flow SDEs monotonically decrease free energy] Let $(p(\cdot, t))_{t \ge T_1}$ be the marginals of a Langevin SDE \eqref{eq:overdamped_langevin_SDE} from the earliest possible observation time in the setting of \cref{cor:three_points}, and suppose that $\Psi \in C^\infty(\R^d)$. 
% , such that $\rho_0$ is a smooth density, with finite second moments and finite entropy, $\int_{\R^d} \rho_0(x) \log \rho_0(x)\dx < \infty$. 
Then, the free energy $J(\rho) = \int_{\R^d} \Psi(x) \rho(x)\dx + \beta\int_{\R^d} \rho(x) \log \rho(x)\dx $, such that $\beta = \frac{\sigma^2}{2}$, strictly decreases in time, unless the process is at its stationary Gibbs distribution, $p_{\mathrm{eq}}$. Thus, $p_{\mathrm{eq}}$ is the only possible marginal that can repeat at distinct times.
\label{prop: decrease_free_energy}
\end{prop} 

\begin{proof}
We apply \cref{lemma: PDE_regularity_finite_time}, which implies that for all $t \ge T_1$, $p(\cdot, t)$ admits a twice differentiable density, with the exponential decay on both the density
$p(x,t) \le Kt^{\frac{-d+2}{2}} \exp\bigl(-\frac{\delta}{2t}\|x\|^2 \bigr)$ and the gradient $\|\nabla p(x,t)\| \le Kt^{\frac{-d+2}{2}} \exp\bigl(-\frac{\delta}{2t}\|x\|^2 \bigr)$, for some $K, \delta > 0$. Thus, we may use chain rule, substitute the Fokker-Planck equation, and integrate by parts (with all boundary terms vanishing) to obtain
\begin{align}
\frac{\text{d}}{\text{d}t} J(p(\cdot,t))
&= \int_{\mathbb{R}^d} \bigl(\Psi(x) + \beta(\log p(x,t) + 1)\bigr)\,\frac{\partial}{\partial t} p(x,t)\, \dx\\
&= \int_{\R^d} \left(\Psi(x) + \beta(\log p(x,t) + 1) \right)\left[\nabla\cdot  (p(x,t) \nabla \Psi(x) ) + \beta\Delta p(x,t) \right]\dx\\
&= -\int_{\R^d} |\nabla \Psi(x)|^2 p(x,t)\dx - 2 \beta \int \nabla p(x,t) \cdot \nabla \Psi(x)\dx  - \beta^2 \int \frac{|\nabla p(x,t)|^2}{p(x,t)}\dx \\
&= - \int_{\R^d}p(x,t)\left\| \nabla \Psi(x) + \beta \nabla \log p(x,t)\right\|^2\dx \le 0,
\end{align}
which has equality only if $\log p(x,t) + \frac{\Psi(x)}{\beta}=K$ for some constant $K$. It follows that the Gibbs distribution
\begin{align*}
p_{\mathrm{eq}}(x) = \frac{1}{Z_{\Psi, \sigma^2}} \exp\left(-\frac{2\Psi(x)}{\sigma^2}\right),
\end{align*}
is the only possible marginal, observed from $t \ge T_1$, which induces no change in free energy $J(p(\cdot,t))$. Moreover, by \cref{lem:finite_free_energy}, we note that $J(p(\cdot, t))<\infty$. Each transient marginal in the evolution $\bigl(p(\cdot, t)\bigr)_{t \ge T_1}$ therefore has a distinct well-defined free energy, which strictly decreases in time, unless we are at $p_{\mathrm{eq}}(x)$. We conclude that the only possible re-occurring marginal in the observation period $t \ge T_1$ is the Gibbs distribution $p_{\mathrm{eq}}$, which is the unique stationary distribution as long as it is integrable \citep{bogachev2022fokker}[Theorem 4.1.11]. 
\end{proof}

\begin{lemma}[Finite free energy]
\label{lem:finite_free_energy}
Let $p(\cdot,t)$ satisfy $p(x,t) \le Kt^{\frac{-d+2}{2}} \exp\bigl(-\frac{\delta}{2t}\|x\|^2 \bigr)$ for some $K, \delta > 0$. Then, $J(p(\cdot,t))=\int_{\R^d}\Psi(x)p(x,t)\dx + \beta \int_{\R^d}p(x,t)\log(p(x,t))\dx <\infty$, with $\beta = \frac{\sigma^2}{2}$.
\begin{proof}
By the fundamental theorem of calculus along the segment $t\mapsto tx$,
\[
\Psi(x)-\Psi(0)=\int_0^1 \nabla\Psi(tx)\cdot x\,\dt.
\]
Then, since we have the linear growth condition, $\|\nabla\Psi(x)\|\le K(1+\|x\|)$ for all $x \in \R^d$, we have
\[
|\Psi(x)|
\le |\Psi(0)| + \|x\|\!\int_0^1\|\nabla\Psi(tx)\|\,\dt
\le |\Psi(0)| + \|x\|\!\int_0^1 K(1+t\|x\|)\,\dt
=|\Psi(0)| + K\|x\| + \tfrac{K}{2}\|x\|^2.
\]
This bound shows that if $p(\cdot,t)$ has finite second moments, then the potential energy term $\int_{\R^d} \Psi(x) p(x,t)\dx$ in the free energy is bounded. Since we always assume that the initial distribution has finite second moments, then $p(\cdot,t)$ also has finite second moments (see \citep[Theorem 1.1.9]{stroock2008partial}). 

We now control the second term in the free energy, by showing that $p(\cdot,t)$ has finite differential entropy. It suffices to apply the Gaussian decay estimate $p(x,t) \le Kt^{\frac{-d+2}{2}} \exp\bigl(-\frac{\delta}{2t}\|x\|^2 \bigr)$. By direct computation, 
\begin{align*}
\int_{\mathbb{R}^d} p(x,t) \log p(x,t)\,\dx
&\le \int_{\mathbb{R}^d} p(x,t) 
\left[ \log K + \tfrac{-d+2}{2}\log t - \tfrac{\delta}{2t} \|x\|^2 \right] \dx\\
&\le \log K \int p(x,t)\,\dx 
+ \tfrac{-d+2}{2}\log t \int p(x,t)\,\dx
- \tfrac{\delta}{2t} \int \|x\|^2 p(x,t)\,\dx \\
&= \log K + \tfrac{-d+2}{2} \log t - \tfrac{\delta}{2t} \mathbb{E}_{p(\cdot,t)}[\|x\|^2] < \infty,
\end{align*}
since \(p(\cdot,t)\) is a probability density with finite second moments. We therefore conclude that the free energy
\begin{align*}
J(p(\cdot,t)) = \int_{\mathbb{R}^d} \Psi(x)\,p(x,t)\,\dx + \beta\int_{\mathbb{R}^d} p(x,t)(x) \log p(x,t)\,\dx<\infty
\end{align*}
\end{proof}
\end{lemma}
We emphasize that \cref{prop: decrease_free_energy} is a classical result, which is well known in the stochastic analysis and statistical mechanics literature, where it is commonly referred to as an H-theorem, due to its connection to Boltzmann's second law of thermodynamics \citep{chafai2015boltzmann}. For example, \citep{jordan1998variational} states that given that $\Psi \in C^\infty(\R^d)$ and $p(\cdot,0)$ is a density with finite free energy, then it is known in the literature (e.g. \citep{risken1989fokker} and attached references) that $J(\rho)$ is a monotone functional on the marginals of a gradient-flow SDE \eqref{eq:overdamped_langevin_SDE}, and that $J(\rho)$ is uniquely minimized at the stationary Gibbs distribution. While we have derived a proof of this result, we believe that the statement holds under more relaxed conditions. In particular, we use the assumption $\Psi \in C^\infty(\R^d)$ to derive the regularity and decay estimates from \cref{lemma: PDE_regularity_finite_time}, which we frequently use for our proof. However, we believe that existing results in parabolic elliptic PDE theory should yield the same estimates under milder assumptions. We note that a proof is given in \citep[Sec. 1.7]{chafai2015boltzmann}, provided that the Gibbs distribution is integrable, and \citep{jordan1998variational} notes that the monotonicity holds, even if the Gibbs distribution is not integrable, but it is unclear what conditions are required.

\section{IDENTIFIABILITY IN THE TIME-INHOMOGENEOUS CASE}
\label{sec:time_inhomo}
Theorem \ref{thm:identifiability} assumes that the drift and diffusion terms are time-homogeneous. However, many physical systems have time-inhomogeneous dynamics. For example, in hydrology, extreme events, such as droughts or fires, can cause ``regime shifts" in the underlying functional relationships \citep{runge2019inferring, gunther2024causal,pedersen1995new,seibert2010effects,maina2020watersheds}. We may define this notion for gradient-flow SDEs, such that the drift and diffusion terms change due to regime shifts at specific times $\{t_i\}_{i=0}^{n-1}$ (with the drift remaining smooth and satisfying the growth condition in each regime to ensure well-posedness).

\begin{defn}[Time-inhomogeneous gradient-flow SDE with discrete regimes] A time-inhomogeneous Langevin SDE with discrete regimes, marked by events taking place at times $\{t_i\}_{i=0}^{n-1}$, is given by
\begin{align}
    \dXt{}=-\nabla\Psi(X_t,t_i)\dt{}+\sigma(t_i)\dWt{},
    \label{eq: regime_langevin_sde}
\end{align}
such that $\nabla\Psi(X_t,t)$ and $\sigma(t)^2$ are time-homogeneous within each regime $t \in [t_i, t_{i+1})$.
\end{defn}

Under these time-inhomogeneous dynamics, identifiability, which is defined analogously to the time-homogeneous case, is still guaranteed by transient marginals.

\begin{cor}[Identifiability for time-inhomogeneous gradient-flow SDEs with discrete regimes]
The time-inhomogeneous gradient-flow SDE with discrete regimes~\eqref{eq: regime_langevin_sde} is identifiable from its marginals $\bigl(p(\cdot, t\bigr)_{t \in [0,T]}$ if and only if $p(\cdot, t_i)$ is not a stationary distribution of $\dXt{}=\nabla\Psi(X_t,t_i)\dt{}+\sigma(t_i)\dWt{}$ for each $t_i$.
\label{cor:regime_shift}
\end{cor}

% Let the observation interval $[0,T]$ be partitioned into $n$ distinct regimes, such that $[0,T] = \cup_{i =0}^{n-1}[t_i, t_{i+1})$, and let $\nabla \Psi: \R^d \times [0,T] \to \R^d$ and $\sigma^2: [0,T] \to \R^+$ be time-inhomogeneous but time-homogeneous within each regime $[t_i, t_{i+1})$. Then, $\nabla \Psi(\cdot,t)$ and $\sigma^2(t)$ of the time-inhomogeneous Langevin SDE are identifiable from its marginals $\bigl(p(\cdot, t\bigr)_{t \in [0,T]}$ if and only if $p(\cdot, t_i)$ is not a stationary distribution of $\dXt{}=\nabla\Psi(X_t,t_i)\dt{}+\sigma(t_i)\dWt{}$ for each $t_i$.
% , i.e., if and only if each regime has transient marginals $\bigl(p(\cdot, t\bigr)_{t \in [t_i,t_{i+1})}$.

\begin{proof}
We apply \cref{thm:identifiability} on each regime $[t_i, t_{i+1})$. Since the regimes form a partition of $[0,T]$, this guarantees that there is a unique SDE with the observed marginals $\bigl(p(\cdot, t\bigr)_{t \in [0,T]}$.
\end{proof}

Intuitively, the drift and diffusion within a given regime $[t_i, t_{i+1})$ are identifiable if and only if the regime's initial marginal, $p(\cdot, t_i)$, is not a stationary distribution. However, we note that if the drift or diffusion terms change continuously in time, then it no longer follows that transience completely characterizes identifiability. Indeed, the process could be in an equilibrium state that itself is changing in time, 
as the next example shows. In particular, distinct time-inhomogeneous SDEs may share the same marginals $\bigl(p(\cdot,t)\bigr)_{t \ge 0}$, as long as for each $t\ge0$, the marginal $p(\cdot,t)$ is precisely the stationary distribution of the time-inhomogeneous residual process, evaluated at that time.

% Therefore, \cref{thm:identifiability,cor:regime_shift} are the most that one can say about transient marginals defining identifiability conditions for Langevin SDEs.

% Intuitively, the only issue that can arise is if the marginal at the start of a regime, $p(\cdot, t_i)$, coincides with the stationary distribution of the regime's SDE, and hence leads to a stationary plateau $[t_i, t_{i+1}]$. 

% A simple example is provided below, where the stationary state at time $t$ is zero-mean Gaussian with variance $t$.

% \begin{ex}[Continuously evolving equilibrium due to time-inhomogeneous potential]
% \label{eq: non_identifiability_time_inhomo} The Langevin SDEs $\dXt{} = \dWt{}$ and $\dYt{} = -\frac{Y_t}{t} \dt{} + \sqrt{3}\dWt{}$
% produce the same marginals $p(\cdot, t)=\mathcal{N}(0,t)$ when initialized with $X_0=Y_0=0$ (i.e., $p_0=\delta(x)$). Thus, the marginals evolve in time, but the two SDEs are non-identifiable. We note that $p(\cdot, t)=\mathcal{N}(0,t)$ is also the stationary distribution of their residual Fokker-Planck equation, when evaluated at time $t$, corresponding to the time-inhomogeneous Langevin SDE $\mathrm{d}Z_t = -\frac{Z_t}{t}\dt{} + \sqrt{2}\dWt{}$.
% \end{ex}

\begin{ex}[Continuously evolving equilibrium due to time-inhomogeneous potential]
\label{ex: non_identifiability_time_inhomo} The SDEs 
\begin{align}
    \mathrm{d}X_t &= \mathrm{d}W_t, \\
    \mathrm{d}Y_t &=
    \begin{cases}
         \mathrm{d}W_t, & t = 0, \\
        -\dfrac{Y_t}{t} \, \mathrm{d}t + \sqrt{3} \, \mathrm{d}W_t, & t > 0.
    \end{cases}
\end{align}
produce the same marginals $p(\cdot, t)=\mathcal{N}(0,t)$ when initialized with $X_0=Y_0=0$, i.e., $p_0=\delta(x)$.  
% We note that $p(\cdot, t)=\mathcal{N}(0,t)$ is also the stationary distribution of their residual Fokker-Planck equation, when evaluated at time $t$, corresponding to the time-inhomogeneous Langevin SDE $\mathrm{d}Z_t = -\frac{Z_t}{t}\dt{} + \sqrt{2}\dWt{}$.
\end{ex}
\begin{proof}
Let $t>0$. The respective Fokker-Planck operators for both SDEs are
\begin{align*}
   \mathcal{L}^*_{0, 1}\bigl(p(\cdot, t)\bigr)(x) &= \frac{1}{2}\frac{\partial^2}{\partial x^2} p(x,t) \\
   \mathcal{L}^*_{-\frac{x}{t}, 3}\bigl(p(\cdot, t)\bigr)(x)&= \frac{\partial}{\partial x}\left(p(x,t) \frac{x}{t}\right) + \frac{3}{2}\frac{\partial^2}{\partial x^2}p(x,t)
\end{align*}
Then, suppose that $p(\cdot, t)=\mathcal{N}(0,t)$, such that $p(x,t) = \frac{1}{\sqrt{2\pi t}} \exp(-\frac{x^2}{2t})$. Then,
\begin{align*}
    \frac{\partial}{\partial x}p(x,t) &= -\frac{x}{t}p(x,t) \\
    \frac{\partial^2}{\partial x^2}p(x,t) &= -\frac{p(x,t)}{t} + \frac{x^2}{t^2}p(x,t)
\end{align*}
It follows that
\begin{align*}
    \mathcal{L}^*_{0, 1}\bigl(p(\cdot, t)\bigr)(x)&=  -\frac{p(x,t)}{2t} + \frac{x^2}{2t^2}p(x,t)\\
    &= \frac{p(x,t)}{2t^2}\bigl(x^2-t \bigr)\\
   \mathcal{L}^*_{-\frac{x}{t}, 3}\bigl(p(\cdot, t)\bigr)(x)&= 
   -\frac{x^2}{t^2}p(x,t) + \frac{p(x,t)}{t} + \frac{3x^2}{2t^2}p(x,t) - \frac{3p(x,t)}{2t}\\
   &= \frac{p(x,t)}{2t^2}\bigl(-2x^2+2t+3x^2 -3t \bigr)\\
   &= \frac{p(x,t)}{2t^2}\bigl(x^2-t\bigr).
\end{align*}
Thus, the Fokker-Planck operators are equivalent for $p(\cdot, t) = \mathcal{N}(0,t)$, which are the marginals of Brownian motion. Equivalently, we observe that $\mathcal{N}(0,t)$ is the stationary distribution of the residual Fokker-Planck equation at time $t$:
\begin{align*}
   0 &= \mathcal{L}^*_{-\frac{x}{t}, 2}\bigl(p_\mathrm{eq}\bigr)(x)= \frac{\partial}{\partial x}\left(p_\mathrm{eq}(x) \frac{x}{t}\right) +\frac{\partial^2}{\partial x^2}p_\mathrm{eq}(x).
\end{align*}
Given $t>0$, $p_\mathrm{eq}(x) = \frac{1}{\sqrt{2\pi t}} \exp(-\frac{x^2}{2t})$ solves the above equation. This yields the interpretation that the marginals are at equilibrium (for the residual process  $\mathrm{d}Z_t = -\frac{Z_t}{t}\dt{} + \sqrt{2}\dWt{}$), which itself is time-varying.
\end{proof}

\section{ADDITIONAL NON-IDENTIFIABILITY EXAMPLES}
\label{sec: non_iden_examples}

For gradient-flow SDEs, we have proven that non-identifiability from marginals arises if and only if the observed marginals also correspond to the marginals of a residual process in equilibrium. An example for the time-homogeneous case was given in \cref{ex:non_iden_stationary}, such that the equilibrium marginals are constant, and an example for the time-inhomogenous case was given in \cref{ex: non_identifiability_time_inhomo}, such that the equilibrium distribution of the residual process continuously changes in time. 

However, for other forms of SDEs, other types of non-identifiability have been documented in the literature, including: undetectable rotations \citep{weinreb2018fundamental, shen2025multi, guan2024identifying}, rank-degenerate trajectories \citep{wang2024generator}, and sharing the same stationary distribution \citep{lavenant2021towards}. We overview examples below:
% We note that the gradient flow drift of a Langevin SDE rules out rotational non-identifiability. Furthermore, isotropic diffusion rules out rank-degeneracy, as well as the case where only certain subspaces may be stationary. In the following, we recall known examples of these scenarios:
\begin{ex}[Gaussian pancake~\protect{\citep{guan2024identifying}}]
Let $a,b,c \in \R$, with $a \leq b$, and let $X_0$ be defined as a $2d$ Gaussian pancake, such that for each fixed $x_0^{(2)} \in [a,b]$, the first coordinate $x_0^{(1)} \sim \mathcal{N}(0, 1)$:
\label{ex: non_identifiability_gaussian_pancakes}
\begin{align}
    \mathrm{d}X_t &= \begin{bmatrix} -1 & 0 \\ 0 & b  \end{bmatrix}\mathrm{d}t + \begin{bmatrix} 1 & 0 \\ 0 & c  \end{bmatrix} \mathrm{d}W_t, &X _0 \sim \mathcal{N}(0, 1) \times [a, b] \\
    \mathrm{d}Y_t &= \begin{bmatrix} -10 & 0 \\ 0 & b  \end{bmatrix}\mathrm{d}t + \begin{bmatrix} \sqrt{10} & 0 \\ 0 & c  \end{bmatrix} \mathrm{d}W_t. &Y _0 \sim \mathcal{N}(0, 1) \times [a, b]
\end{align}
\end{ex}

\begin{ex}[Rotation \protect{\citep{shen2025multi,hashimoto2016learning,weinreb2018fundamental, guan2024identifying}}]
\label{eq: non_identifiability_rotation}
\begin{align*}
    \mathrm{d}X_t &= \mathrm{d}W_t, &X_0 \sim \mathcal{N}(0, I)\\
    \mathrm{d}Y_t &= \begin{bmatrix} 0 & 1 \\ -1 & 0  \end{bmatrix}Y_t \: \mathrm{d}t
    + \mathrm{d}W_t. &  Y_0 \sim \mathcal{N}(0, I)
\end{align*}
\end{ex}
\begin{ex}[Degenerate rank \protect{\citep{wang2024generator, guan2024identifying}}]
\label{eq: non_identifiability_wang}
\begin{align*}
    \mathrm{d}X_t &= \begin{bmatrix} 1 & 2 \\ 1 & 0  \end{bmatrix} X_t \: \mathrm{d}t
    +  \begin{bmatrix} 1 & 2 \\ -1 & -2  \end{bmatrix}\mathrm{d}W_t, &X_0 = \begin{bmatrix}
    1 \\ -1
    \end{bmatrix}\\
    \mathrm{d}Y_t &=  \begin{bmatrix} 1/3 & 4/3 \\ 2/3 & -1/3  \end{bmatrix} Y_t \: \mathrm{d}t
    + \begin{bmatrix} 1 & 2 \\ -1 & -2  \end{bmatrix}\mathrm{d}W_t. &Y_0 = \begin{bmatrix}
    1 \\ -1
    \end{bmatrix}
\end{align*}
\end{ex}

\section{ADDITIONAL NOTES ON NN-APPEX}
\label{sec: supplement_sb_refinement}

\subsection{Convergence}
To discuss convergence of \nappex's tri-level iterative scheme \eqref{eq: SB_traj_inf_k}-\eqref{eq: diff_MLE_k}, we first rigorously justify why we may re-estimate diffusion while ensuring that the objective remains finite. In particular, we must show that we have finite $\mathrm{KL}$ divergence between the reconstructed paths from the step \eqref{eq: SB_traj_inf_k} and the paths following MLE parameter estimation of a gradient-flow SDE from the steps \eqref{eq: drift_MLE_k} and \eqref{eq: diff_MLE_k}. 

Given continuously observed marginals of a $d$-dimensional process, recall that the $\mathrm{KL}$ divergence between two laws on paths $P$ and $Q$, taken over the path space $\Omega = C([0,T], \R^d)$, is given by
\begin{align}
    \mathrm{KL}(Q\|P) = \int_{\Omega} \log\left(\frac{dQ}{dP}(\omega) \right)dQ(\omega).
    \label{eq: KL_cts}
\end{align}
By Girsanov's theorem, \eqref{eq: KL_cts} is only finite for $Q$ and $P$, if their underlying SDEs share the same diffusion \citep{vargas2021solving}. Besides non-identifiability, this is another cited reason for why previous Schr\"odinger Bridge methods assume that diffusivity must be fixed \citep{shen2025multi}. However, in the practical setting, we only observe a finite number of marginals, so we should instead optimize $\mathrm{KL}$ divergence over the discretized path space $\Omega_N = C(\{t_i\}_{i=0}^{N-1}, \R^d)$. Then, for each law on paths, we only observe the couplings, $Q_{t_{i}, t_{i+1}}$ and $P_{t_{i}, t_{i+1}}$, between consecutive times. Let $Q^N$ and $P^N$ denote the concatenations of these couplings from $t_0$ to $t_{N-1}$. Then, by \citet{benamou2019entropy}[Lemma 3.4], evaluating the $\mathrm{KL}$ divergence over $\Omega_N = C(\{t_i\}_{i=0}^{N-1}, \R^d)$ yields 
\begin{align}
    \mathrm{KL}(P^N\|Q^N) = \sum_{i=0}^{N-2}\mathrm{KL}(P_{t_{i}, t_{i+1}}\| Q_{t_{i}, t_{i+1}}) - \sum_{i=1}^{N-2}\mathrm{KL}(P_{t_{i}}\| Q_{t_{i}}).
    \label{eq: KL_discrete}
\end{align}
First note that by the data processing inequality, $\mathrm{KL}(P_{t_{i}}\| Q_{t_{i}}) \le \mathrm{KL}(P_{t_{i}, t_{i+1}}\| Q_{t_{i}, t_{i+1}})$. Hence, the expression will be finite as long as $\mathrm{KL}(P_{t_{i}, t_{i+1}}\| Q_{t_{i}, t_{i+1}}) = \int_{\R^d \times \R^d} \log\left(\frac{dP_{t_{i}, t_{i+1}}}{dQ_{t_{i}, t_{i+1}}}\right)P_{t_{i}, t_{i+1}}(x)dx<\infty$. Since the transition density of any nondegenerate diffusion process is absolutely continuous with respect to the Lebesgue measure for any positive time, it follows that the expression is finite for all gradient-flow SDEs \eqref{eq:overdamped_langevin_SDE}. 

As described in \cref{sec:method}, \nappex{} will alternate between finding a law on paths $P$ that minimizes KL in the first argument (trajectory inference) and finding a law on paths $Q$ that minimizes KL in the second argument (MLE parameter estimation). Note that each optimization is subject to distinct hard constraints, which respectively stipulate that $P \in \Pi(p(\cdot, t_i)_{i=0}^{N-1})$ adheres to the produced marginals and that $Q$ is the law of a gradient-flow SDE. The algorithm therefore performs alternating projections, onto distinct spaces. In fact, by \cref{cor:three_points}, only one gradient-flow SDE $Q_{\shortminus\nabla\Psi, \sigma^2}$ obeys the marginal constraints, which implies that the intersection of the two projection spaces is a singleton.  In other words,
$$\inf\limits_{\substack{
    P \in \Pi\bigl(p(\cdot, t_i)_{i=0}^{N-1}\bigr) 
    % \\
    % (\shortminus\nabla\Psi,\sigma^2) \in \mathcal{S}
}}
\mathrm{KL}(P \| Q_{\shortminus\nabla\Psi, \sigma^2}) = 0 \iff P = Q_{\shortminus\nabla\Psi, \sigma^2}.
$$
By construction, the iterates will produce a monotonically decreasing sequence of $\mathrm{KL}$ divergences between reconstructed laws and MLE estimated laws. With the above argument, we have that the $\mathrm{KL}$ divergences are bounded above by some $M>0$, and from below by $0$ (attained by the true gradient-flow SDE parameters $(-\nabla \Psi, \sigma^2)$). Thus, the sequence of $\mathrm{KL}$ divergences must converge. However, as noted in \citet{shen2025multi}, this does not imply that the arguments necessarily converge, since multiple pairs $(P,Q)$ may produce the same divergence, such that the iterative algorithm gets stuck in a cycle. While this has not been observed empirically, and is impossible if the divergence is $0$ (since there is a unique gradient-flow SDE satisfying the marginal constraints), it is an open question whether the iterative scheme will converge to the unique optimal parameters. Indeed, while convergence of the iterative SB refinement algorithm has been proven in the case where the family of admissible probability transition densities is convex \citep[Proposition 1]{shen2025multi}, this condition does not generally hold, and is false for the family of gradient-flow SDEs.

\subsection{Runtime}
The runtime of \nappex{} is split across its three subprocedures: 
\begin{enumerate}
    \item \textbf{Trajectory inference \eqref{eq: SB_traj_inf_k}}: solving the multi-marginal SB problem with respect to the current reference SDE to infer and sample from a law on paths.
    \item \textbf{MLE drift estimation \eqref{eq: drift_MLE_k}}: training a neural network whose parameters minimize the objective \eqref{eq:likelihood_fn_nn} over the inferred paths.
    \item \textbf{MLE diffusion estimation \eqref{eq: diff_MLE_k}}: computing the quadratic variation \eqref{eq:quad_var_diff_MLE} over the inferred paths, conditioned on the drift estimate.
\end{enumerate}
We note that \nappex's runtime will vary for different implementations of these subprocedures. In particular, there are multiple approaches for solving the multi-marginal SB problem. One approach is to consider all contiguous pairs of marginals, and then to apply iterative proportional fitting (Sinkhorn's algorithm) on each pair to determine the distribution over couplings \citep{shen2025multi, guan2024identifying}. To enforce consistency across transitions, we instead apply a multi-marginal iterative proportional fitting across all marginals, see \citep{marino2020optimal}[(4.9)]. This algorithm is analogous to the two-marginals setting, but it instead rescales per-time slices for each marginal, rather than just the endpoints. We observed better accuracy for this approach, given the same stopping criteria and maximum number of iterations, and were thus able to reduce runtime with relatively fewer iterations. We similarly note that  MLE drift estimation can be implemented using different neural network architectures and hyperparameters. In all experiments we use $128$ neurons per layer, but this can be changed with some alterations to both the flexibility and the runtime (Table \ref{tab:appex_timing_nnsize}).

\begin{table}[h]
    \centering
    \caption{Average runtime of drift estimation step \eqref{eq: drift_MLE_k} using different neural network widths.}
    \label{tab:appex_timing_nnsize}
    \begin{tabular}{cc}
        \toprule
        \textbf{NN Width} & \textbf{Avg Time (s)} \\
        \midrule
        16  & 1.32  \\
        \midrule
        32  & 1.85  \\
        \midrule
        64  & 3.20  \\
        \midrule
        128 & 6.57  \\
        \midrule
        256 & 13.20 \\
        \bottomrule
    \end{tabular}
\end{table}

In contrast to the drift estimation, the diffusion MLE estimate is a closed formula, and is thus cheap to compute. We report the average runtime (in seconds) of each of the three subprocedures for a single iteration, run on our main experiments in Table \ref{tab:appex_timing_overall}.

\begin{table}[htbp!]
\centering
\small
\caption{Average runtime per iteration (seconds) for each subprocedure, aggregated across all potentials.}
\begin{tabular}{lcccc}
\toprule
& \textbf{Trajectory inference} & \textbf{Drift MLE} & \textbf{Diffusion MLE} & \textbf{Total} \\
\midrule
Average & 3.842 $\pm$ 0.046 & 6.176 $\pm$ 0.061 & 0.011 $\pm$ 0.002 & 10.030 $\pm$ 0.065 \\
\bottomrule
\end{tabular}
\label{tab:appex_timing_overall}
\end{table}

Since we used $30$ iterations of \nappex{}, each SDE inference for our main experiments was approximately $5$ minutes. Since the diffusion MLE is computationally negligible, \nappex{} has virtually the same runtime as an analogous \sbirr{} algorithm. In contrast, \wot{} only comprises a single iteration, and thus has the fastest runtime.

Stopping criteria, such as thresholds on the drift and diffusion estimates of the iterates, or approximations of the $\mathrm{KL}$ objective based on the $W^1$,$W^2$ metrics or maximum mean discrepancy (MMD),  can be enforced for better performance and faster runtime. However, we note that these evaluations tend to be computationally intensive themselves, since they either require evaluations on a discretized grid or an approximation of relative entropy.

\section{MAXIMUM LIKELIHOOD ESTIMATION FOR DRIFT AND DIFFUSION}
\label{sec: MLE_appendix}

Let the parameters of a gradient-flow SDE \eqref{eq:overdamped_langevin_SDE} be given by drift $-\nabla \Psi_\theta(x)$ and diffusivity $\sigma^2$. Then, if we apply the first-order Euler-Maruyama linear approximation, we have that $X_{t+\Delta t} | X_t \sim \mathcal{N}(X_t - \nabla \Psi_\theta(X_t)\Delta t, \sigma^2 \Delta t)$. Hence, for a single observation of a trajectory $x_t \to x_{t+\Delta t}$,
\begin{align*}
    p(x_{t+\Delta t}|x_t) = \frac{1}{(2\pi \sigma^2\Delta t)^{d/2}}\exp\left(-\frac{\|x_{t+\Delta t} - x_t + \nabla \Psi_\theta(x_t)\Delta t\|^2}{2\sigma^2\Delta t}\right).
\end{align*}
The full negative log-likelihood function is therefore given by
\begin{align}
    l(\theta, \sigma^2 | x_t, x_{t+\Delta t}) = -\log\left(p(x_{t+\Delta t}|x_t)\right) = \frac{d}{2}\log(2\pi\Delta t) + \frac{d}{2}\log(\sigma^2) + \frac{\|x_{t+\Delta t} - x_t + \nabla \Psi_\theta(x_t)\Delta t\|^2}{2\sigma^2\Delta t}.
\end{align}
The MLE parameters $(\hat{\theta}, \hat{\sigma}^2)$ minimize $l(\theta, \sigma^2 | x_t, x_{t+\Delta t})$. However, note that $\theta$ only appears in the numerator of the last term. It follows that $\hat{\theta}$ minimizes the squared error of the finite difference, which we denote
\begin{align}
    \ell(\theta) = \|x_{t+\Delta t} - x_t + \nabla \Psi_\theta(x_t)\Delta t\|^2.
\end{align}
If we observe $M$ independent trajectories over $N-1$ time steps ($N$ observed times), let $X= \bigl\{x_{i\Delta t}^{(m)}, x_{(i+1)\Delta t}^{(m)}: m \in \{1, \cdots, M \},i \in \{0, \cdots, N-2\}\bigr\}$. We similarly observe that the negative log-likelihood function is given by 
\begin{align*}
     l\left(\theta, \sigma^2 | X \right) = -\log\left(\prod_{m=1}^{M}\prod_{i=0}^{N-2} p(x_{(i+1)\Delta t}^{(m)}|x_{i\Delta t}^{(m)})\right) = -\sum_{m=1}^M\sum_{i=0}^{N-2}\log(p(x_{(i+1)\Delta t}^{(m)}|x_{i\Delta t}^{(m)})),
\end{align*}
and it follows that $\hat{\theta}$ would minimize the mean squared error
\begin{align}
    \ell(\theta)=\sum_{m=1}^{M}\sum_{i=0}^{N-2}\Big\|x_{(i+1)\Delta t}^{(m)}-x_{i\Delta t}^{(m)}+\nabla \Psi_\theta(x_{i\Delta t}^{(m)})\Delta t\Big\|_2^2.
\end{align}
To derive the diffusion MLE estimator, $\hat{\sigma}^2$, we first fix $\theta = \hat{\theta}$, since the drift parameter can be independently optimized. Then, we solve $\frac{\partial}{\partial \sigma^2}l\left(\hat{\theta}, \sigma^2 | X\right) =0$, and obtain
\begin{align*}
    0 &= \sum_{m=1}^{M}\sum_{i=0}^{N-2}\frac{d}{2\sigma^2} - \frac{1}{2\sigma^4 \Delta t}\sum_{m=1}^{M}\sum_{i=0}^{N-2}\Big\|x_{(i+1)\Delta t}^{(m)}-x_{i\Delta t}^{(m)}+\nabla \Psi_{\hat{\theta}}(x_{i\Delta t}^{(m)})\Delta t\Big\|_2^2 \\
    0 &= dM(N-1)- \frac{1}{ \sigma^2\Delta t} \ell(\hat{\theta})\\
    \sigma^2 &= \frac{1}{d\,M\,(N-1)\,\Delta t}\ell(\hat{\theta})
\end{align*}

\section{POTENTIALS}
\label{sec:potentials}

Our experiments consider the following potentials for simulating gradient-flow SDEs. In particular, these are the potentials from \citet{terpin2024learning}, which admit a valid stationary Gibbs distribution $p_{\mathrm{eq}}=\frac{1}{Z}\exp\left(-\frac{\Psi}{2\sigma^2} \right)$.
\begin{align}
\text{\textbf{Bohachevsky}} \qquad
\Psi(x) &= 10\!\left(x_1^2 + 2x_2^2 - 0.3\cos(3\pi x_1) - 0.4\cos(4\pi x_2)\right)
\label{eq:pot-bohachevsky} \\[1em]
\text{\textbf{Oakley--O'Hagan}} \qquad
\Psi(x) &= 5 \sum_{i=1}^2 \left(\sin(x_i) + \cos(x_i) + x_i^2 + x_i\right)
\label{eq:pot-oakley-ohagan} \\[1em]
\text{\textbf{Quadratic}} \qquad
\Psi(x) &= 5\,\|x\|^2
\label{eq:pot-quadratic} \\[1em]
\text{\textbf{Styblinski--Tang}} \qquad
\Psi(x) &= \tfrac{1}{2}\sum_{i=1}^2 \left(x_i^4 - 16x_i^2 + 5x_i\right)
\label{eq:pot-styblinski-tang} \\[1em]
\text{\textbf{Wavy plateau}} \qquad
\Psi(x) &= \sum_{i=1}^2 \left( \cos(\pi x_i) + \tfrac{1}{2}x_i^4 - 3x_i^2 + 1 \right)
\label{eq:pot-wavy-plateau}
\end{align}
Each of these potentials is smoothly differentiable, and we note that they each lead to well-defined strong solutions to the SDE \eqref{eq:overdamped_langevin_SDE}. Indeed, in order to satisfy the linear growth condition, $\|\nabla \Psi(x)\| \le K(1+\|x\|)$, we can multiply potentials by a smooth cutoff function, which is $1$ in $\|x\| \le M$ and $0$ in $\|x \| > 2M$. Choosing $M=100$ for our experiments does not alter data generation compared to the raw setting. Similarly, since the cutoff function ensures that $\Psi$ is supported in a compact subset of $\R^d$, it follows that smoothness of $\Psi$ implies that the drift $-\nabla \Psi$ is Lipschitz.

\section{ADDITIONAL EXPERIMENTS}
\label{sec: additional_experiments_appendix}

\paragraph{Software and hardware details.} The code in this paper was adapted from two public code repositories (\texttt{APPEX}: \url{https://github.com/guanton/APPEX} and \jkonetstar{}: \url{https://github.com/antonioterpin/jkonet-star}). All computations are performed on a 2024 MacBook Pro with $16$GB RAM and an Apple M4 chip. The code was adapted to design the experiments, visualize the data, and interpret the results. It is available in the supplementary material.

% In this section, we consider the same experimental setup from \cref{sec: main_experiments}, and report parameter estimation results based on the remaining confining potentials from \citep{terpin2024learning}: 
% \begin{align}
%     \Psi_{\text{Bohachevsky}}(x) &= 10\bigl(x_1^2 + 2x_2^2 - 0.3\cos(3\pi x_1) - 0.4\cos(4\pi x_2)\bigr),
%     \label{eq:bohachevsky_potential} \\
%     \Psi_{\text{Wavy plateau}}(x) &=  \sum_{i=1}^2 \left( \cos(\pi x_i) + \frac{1}{2}x_i^4 -3x_i^2 + 1 \right),
%     \label{eq:wavy_plateau_potential}
%     \\
%     \Psi_{\text{Oakley-Ohagan}}(x) &= 5 \sum_{i=1}^2 \left(\sin(x_i) + \cos(x_i) + x_i^2 + x_i \right).
%     \label{eq:oakley-ohagan_potential}
% \end{align}
\subsection{Experimental Setup}
Since the objective of our main experiments was to evaluate the inferential power of different Schr\"odinger Bridge methods, we considered the methods \wot, \sbirr, and \nappex{} (ours), and we simulated data from a range of population dynamics for each SDE, by randomizing the initial distribution over Gaussian mixtures. We perform two additional experiments, such that we use the same experimental setup, but consider two particular initial distributions: the uniform distribution over the region of interest $[-4,4]^2$ and the stationary Gibbs distribution $p_{\mathrm{eq}}$. In order to sample from the Gibbs distributions $\frac{1}{Z}\exp(-\frac{2\Psi(x)}{\sigma^2})$, we generate trajectories from $p_0= \textrm{Unif}([-4,4]^2)$ and run $200$ steps of the SDE with step size $\Delta t=0.01$. Indeed, each of the potentials \eqref{eq:pot-bohachevsky}-\eqref{eq:pot-wavy-plateau} satisfies the Poincaré inequality, which ensures that marginals converge exponentially quickly to the stationary Gibbs distribution \citep[Theorem 4.4]{pavliotis2014stochastic}. 

We also perform additional experiments comparing our method, \nappex, against the state-of-the-art variational method, \jkonetstar{} \citep{terpin2024learning}. We implement \jkonetstar{} with default hyperparameters, and use $100$ training epochs, noting rapid convergence of the objective. We verified in a separate experiment that running with $10,000$ epochs did not meaningfully change performance. The experimental setting is the main setting over random GMM initializations, in order to test on different population dynamics. Since \jkonetstar{} also jointly estimates drift and diffusion for the gradient-flow SDE \eqref{eq:overdamped_langevin_SDE}, we additionally report diffusivity MAE for this experiment.

\subsection{Results}
The uniform initial distribution $p_0\sim \mathrm{Unif}[-4,4]^2$ represents an idealized benchmark \citep{terpin2024learning} while the stationary distribution $p_0 \sim p_{\mathrm{eq}}$ represents the non-identifiable setting. Results comparing the three SB methods, \wot, \sbirr, and \nappex, are summarized for the uniform initialization in \cref{fig:boxplots_unif} and for the non-identifiable stationary initialization in \cref{fig:boxplots_gibbs}. 

We also present results from our main experimental setting (random GMM initializations), which compare our method \nappex{} against \jkonetstar{} in \cref{fig:boxplots_jko_appex}. The results show that \nappex{} accurately infers both the drift and diffusivity to high precision for all SDEs, and outperforms \jkonetstar except for the diffusivity estimate of the Bohachevsky potential. We note that the Bohachevsky potential is particularly difficult to estimate, as it has the highest magnitude, and converges to its Gibbs distribution $p_{\mathrm{eq}}$ at a much faster rate than the other SDEs, all while exhibiting highly periodic and nonlinear dependencies. For this reason, the MLE diffusion estimate by \nappex{} likely achieves biased estimates, since the adjusted quadratic variation \eqref{eq:quad_var_diff_MLE} is impacted by poor drift estimation. This is also consistent with previous results \citep{terpin2024learning}, which showed that the Bohachevsky potential is particularly difficult to estimate in dimension $d=2$.

Tables \ref{tab:gradient_magnitude} and \ref{tab:performance_comparison_mean_median} provide further details about our simulated experimental results from the main text. Table \ref{tab:gradient_magnitude} shows that the performance gap between \sbirr{} and \nappex{} is highest in low potential areas, i.e. regions where the magnitude of the gradient field $\nabla \Psi$ is small. Further, most potentials from this benchmark have high magnitudes in a large portion of the evaluation grid, while \sbirr{} particularly struggles in low potential areas (see \cref{fig:misspecified_diffusion} and \cref{fig:oakley-gmms-landscapes}). Intuitively, if the potential is strong enough, then it can still be approximated even if diffusion is misspecified. To emphasize that \nappex{} is especially useful for inference in low potential zones, \cref{tab:gradient_magnitude} stratifies the results based on the magnitude of the true gradient at each evaluation point. All results show that \nappex{} outperforms \sbirr, especially in low potential zones, as one would expect from Figures \ref{fig:misspecified_diffusion} and \ref{fig:oakley-gmms-landscapes}. This is likely due to diffusion having an outsized noisy impact on the data in these regions. Moreover, Table \ref{tab:performance_comparison_mean_median} shows that the median (used in the main text) is actually a favourable evaluation metric for \sbirr{} compared to the mean (not shown in main text). Indeed, across all potentials, the gap in drift estimation performance between \sbirr{} and \nappex{} is larger for the mean compared to the median. This occurs since \sbirr{} performs fairly well most of the time, but can still frequently catastrophically fail due to poor initial diffusion estimates.

% Bohachevsky potential \eqref{eq:pot-bohachevsky}. This is consistent with previous results \citep{terpin2024learning}, which showed that the Bohachevsky potential is particularly difficult to estimate in dimension $d=2$ for the variational method \jkonetstar. In particular the Bohachevsky potential has the highest magnitude, and converges to its Gibbs distribution $p_{\mathrm{eq}}$ at a much faster rate than the other SDEs, all while exhibiting highly periodic and nonlinear dependencies.

% We also note that all methods struggled when given data derived from Bohachevsky potential. \texttt{APPEX} and \jkonetstar{}$_{\ell}$ clearly struggled to estimate both the drift and diffusion in the Bohachevsky setting (Table \ref{tab:maes_drift_diff_allpots}). While the drift estimation was still better given transient marginals compared to stationary (even at Gibbs points), the diffusion estimates were quite far off in the transient setting. 

\begin{figure*}[!h]
    \centering

    % --- Top subfigure ---
    \begin{subfigure}{\textwidth}
        \centering
        \textbf{(a) Absolute drift error (normalized by true magnitude)}\\[0.3em]
        \includegraphics[width=\textwidth]{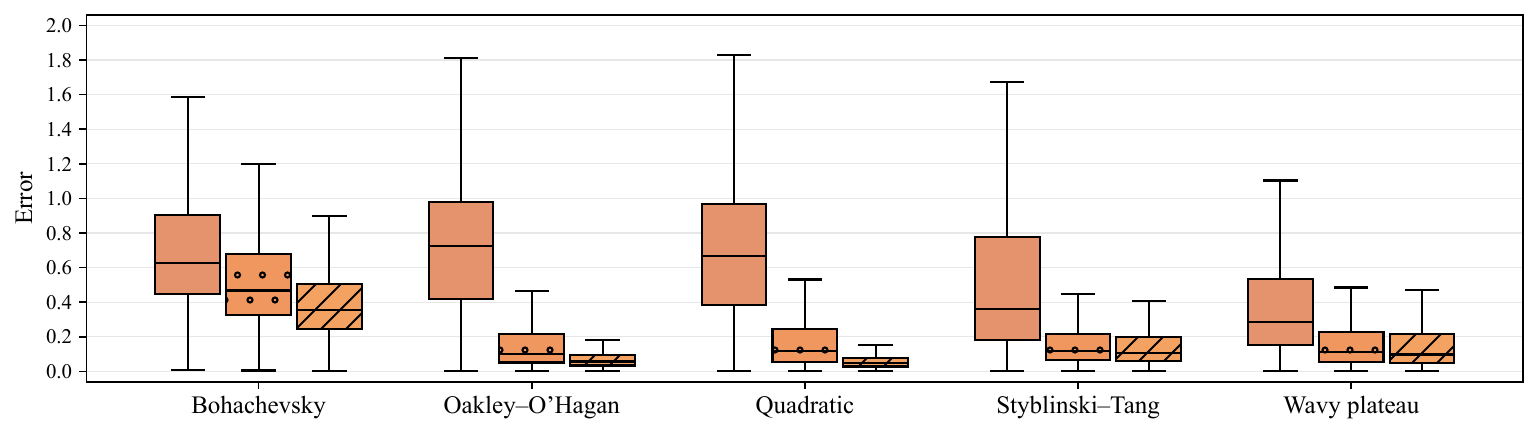}
        \label{fig:box_mae_unif}
    \end{subfigure}

    \vspace{1em}

    % --- Bottom subfigure ---
    \begin{subfigure}{\textwidth}
        \centering
        \textbf{(b) Cosine similarity}\\[0.3em]
        \includegraphics[width=\textwidth]{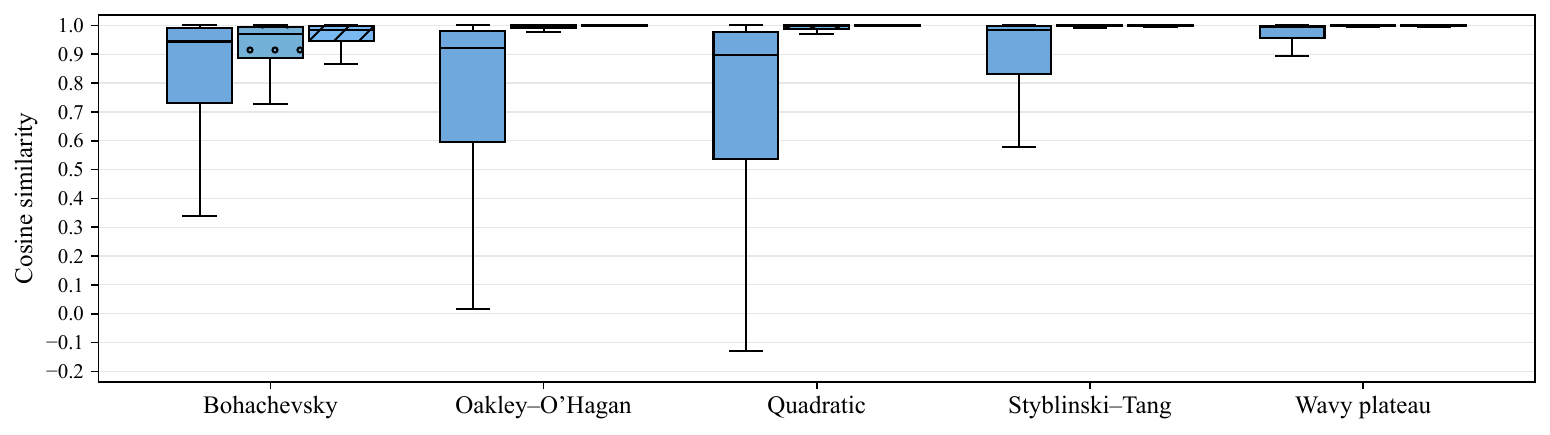}
        \label{fig:box_cos_unif}
    \end{subfigure}

    \vspace{1em}

    % --- Legend centered below ---
    \begin{tikzpicture}
      % sizes
      \def\w{1}   % rectangle width (cm)
      \def\h{0.4} % rectangle height (cm)
      \def\dx{4.0} % horizontal spacing (cm)
      \def\labgap{0.2} % label gap (cm)
      \def\hatchdist{4pt}
      \def\hatchwidth{0.3pt}
      \def\hatchangle{45}

      % WOT
      \begin{scope}[shift={(0,0)}]
        \draw[black] (0,0) rectangle (\w,\h);
        \node[anchor=west] at (\w+\labgap,\h/2) {\wot};
      \end{scope}

      % SBIRR
      \begin{scope}[shift={(\dx,0)}]
        \draw[black] (0,0) rectangle (\w,\h);
        \draw[black] (0.15,0.5*\h) circle (0.02);
        \draw[black] (0.50,0.5*\h) circle (0.02);
        \draw[black] (0.85,0.5*\h) circle (0.02);
        \node[anchor=west] at (\w+\labgap,\h/2) {\sbirr};
      \end{scope}

      % nn-APPEX
      \begin{scope}[shift={(2*\dx,0)}]
        \path[
          pattern = {Lines[angle=\hatchangle, distance=\hatchdist, line width=\hatchwidth]},
          pattern color = black
        ] (0,0) rectangle (\w,\h);
        \draw[black] (0,0) rectangle (\w,\h);
        \node[anchor=west] at (\w+\labgap,\h/2) {\nappex{} (ours)};
      \end{scope}
    \end{tikzpicture}

    \caption{The ability of different Schr\"odinger Bridge methods to infer the gradient-flow drift is evaluated across five potentials using (a) normalized absolute error and (b) cosine similarity. Here, methods observe three marginals with a uniform initial distribution in the region of interest, i.e. $p_0 \sim \mathrm{Unif}[-4,4]^2$. Box-and-whisker plots over $10$ seeds show \nappex{} performs best across all potentials.}
    \label{fig:boxplots_unif}
\end{figure*}

\begin{table}[!h]
    \centering
    \caption{Mean difference in normalized absolute drift error (and cosine similarity) between \nappex{} and \sbirr. Results are displayed for each potential and aggregated with respect to low potential zones (gradient magnitude is less than median) and high potential zones (gradient magnitude is higher than median). Higher numbers indicate a greater outperformance of \nappex{} over \sbirr.}
    \label{tab:gradient_magnitude}
    \begin{tabular}{lcc}
        \toprule
        & \multicolumn{2}{c}{\textbf{Gradient Magnitude}} \\
        \cmidrule(lr){2-3}
        \textbf{Potential} & \textbf{Low} & \textbf{High} \\
        \midrule
        Bohachevsky      & 0.293 (0.074) & 0.184 (0.018) \\
        Oakley-O'Hagan  & 0.612 (0.100) & 0.253 (0.042) \\
        Quadratric        & 1.397 (0.098) & 0.654 (0.054) \\
        Styblinski-Tang & 0.128 (0.031) & 0.075 (0.005) \\
        Wavy Plateau    & 2.199 (0.035) & 0.185 (0.012) \\
        \bottomrule
    \end{tabular}
\end{table}

\begin{figure*}[t]
    \centering

    % --- Top subfigure ---
    \begin{subfigure}{\textwidth}
        \centering
        \textbf{(a) Absolute drift error (normalized by true magnitude)}\\[0.3em]
        \includegraphics[width=\textwidth]{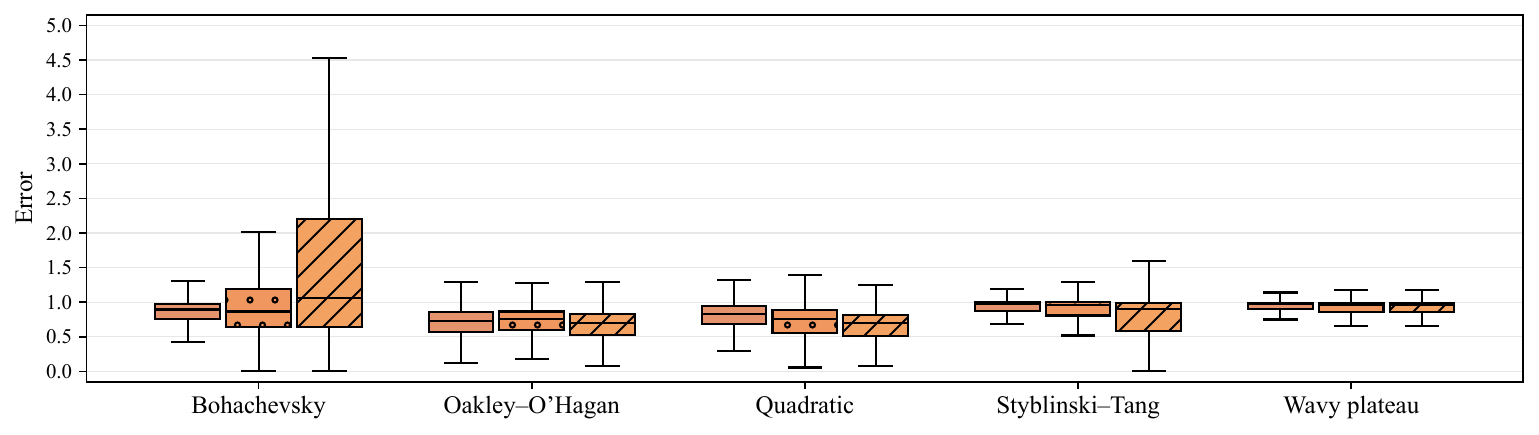}
        \label{fig:box_mae_gibbs}
    \end{subfigure}

    \vspace{1em}

    % --- Bottom subfigure ---
    \begin{subfigure}{\textwidth}
        \centering
        \textbf{(b) Cosine similarity}\\[0.3em]
        \includegraphics[width=\textwidth]{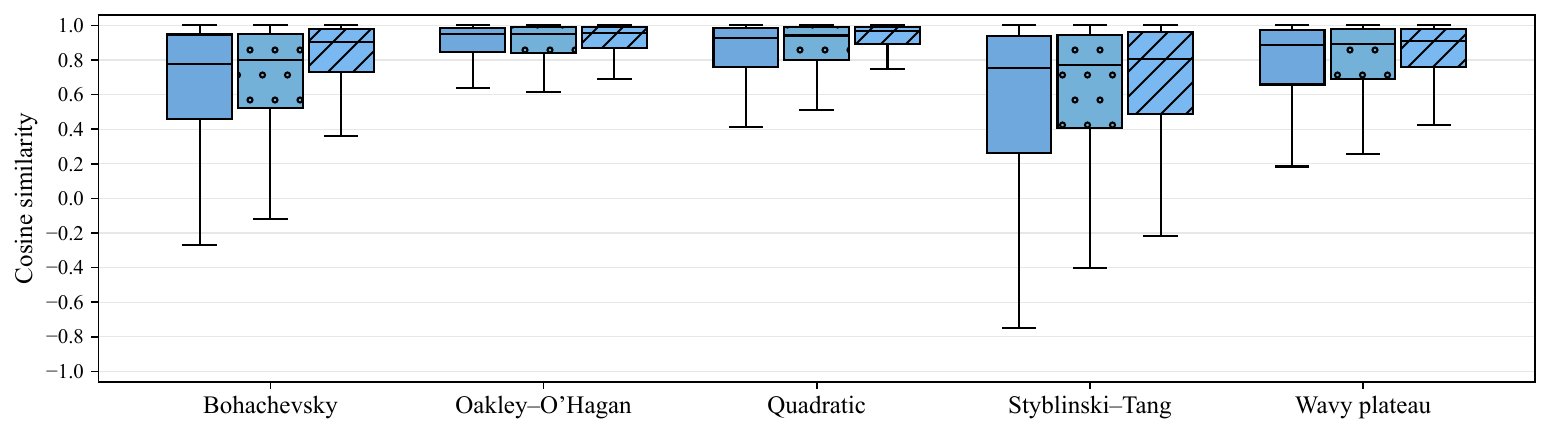}
        \label{fig:box_cos_gibbs}
    \end{subfigure}

    \vspace{1em}

    % --- Legend centered below ---
    \begin{tikzpicture}
      % sizes
      \def\w{1}   % rectangle width (cm)
      \def\h{0.4} % rectangle height (cm)
      \def\dx{4.0} % horizontal spacing (cm)
      \def\labgap{0.2} % label gap (cm)
      \def\hatchdist{4pt}
      \def\hatchwidth{0.3pt}
      \def\hatchangle{45}

      % WOT
      \begin{scope}[shift={(0,0)}]
        \draw[black] (0,0) rectangle (\w,\h);
        \node[anchor=west] at (\w+\labgap,\h/2) {\wot};
      \end{scope}

      % SBIRR
      \begin{scope}[shift={(\dx,0)}]
        \draw[black] (0,0) rectangle (\w,\h);
        \draw[black] (0.15,0.5*\h) circle (0.02);
        \draw[black] (0.50,0.5*\h) circle (0.02);
        \draw[black] (0.85,0.5*\h) circle (0.02);
        \node[anchor=west] at (\w+\labgap,\h/2) {\sbirr};
      \end{scope}

      % nn-APPEX
      \begin{scope}[shift={(2*\dx,0)}]
        \path[
          pattern = {Lines[angle=\hatchangle, distance=\hatchdist, line width=\hatchwidth]},
          pattern color = black
        ] (0,0) rectangle (\w,\h);
        \draw[black] (0,0) rectangle (\w,\h);
        \node[anchor=west] at (\w+\labgap,\h/2) {\nappex{} (ours)};
      \end{scope}
    \end{tikzpicture}

    \caption{The ability of different Schr\"odinger Bridge methods to infer the gradient-flow drift is evaluated across five potentials using (a) normalized absolute error (lower is better) and (b) cosine similarity (higher is better). Methods observe three marginals with the initial distribution equal to the SDE's stationary Gibbs distribution, $p_0 \sim p_{\mathrm{eq}}$. Aggregated box-and-whisker plots over $10$ seeds show that all methods perform similarly poorly, corroborating \cref{prop:stationary_non_iden}: the true gradient-flow drift is not identifiable without knowing the diffusivity.}
    \label{fig:boxplots_gibbs}
\end{figure*}

\begin{figure*}[t]
    \centering

    % --- Top subfigure ---
    \begin{subfigure}{\textwidth}
        \centering
        \textbf{(a) Absolute drift error (normalized by true magnitude)}\\[0.3em]
        \includegraphics[width=\textwidth]{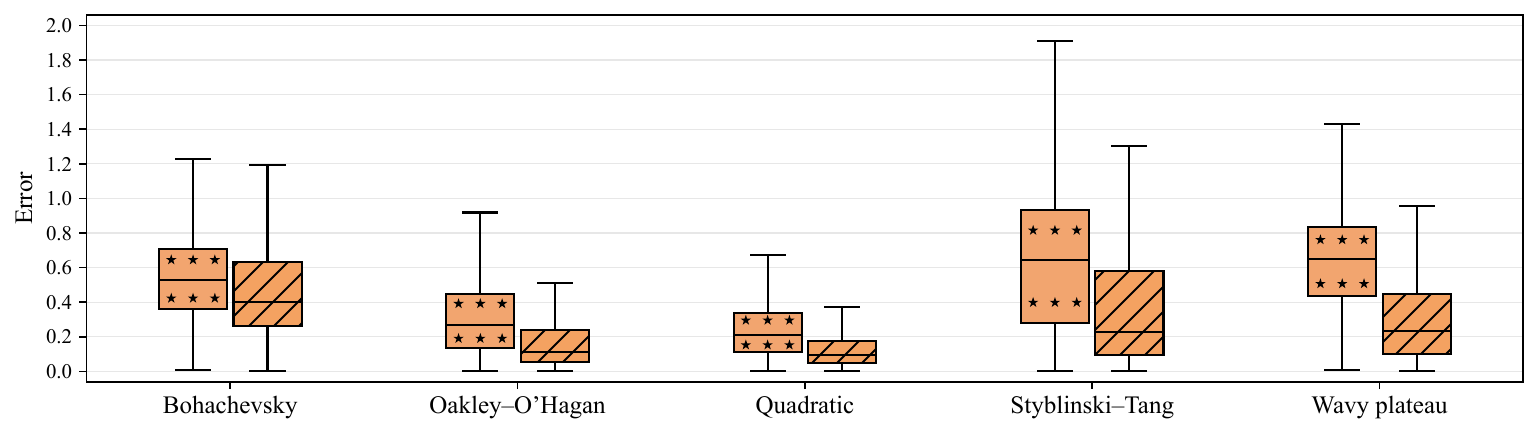}
        \label{fig:box_mae_appex_jkonetstar}
    \end{subfigure}

    \vspace{1em}

    % --- Second subfigure ---
    \begin{subfigure}{\textwidth}
        \centering
        \textbf{(b) Cosine similarity}\\[0.3em]
        \includegraphics[width=\textwidth]{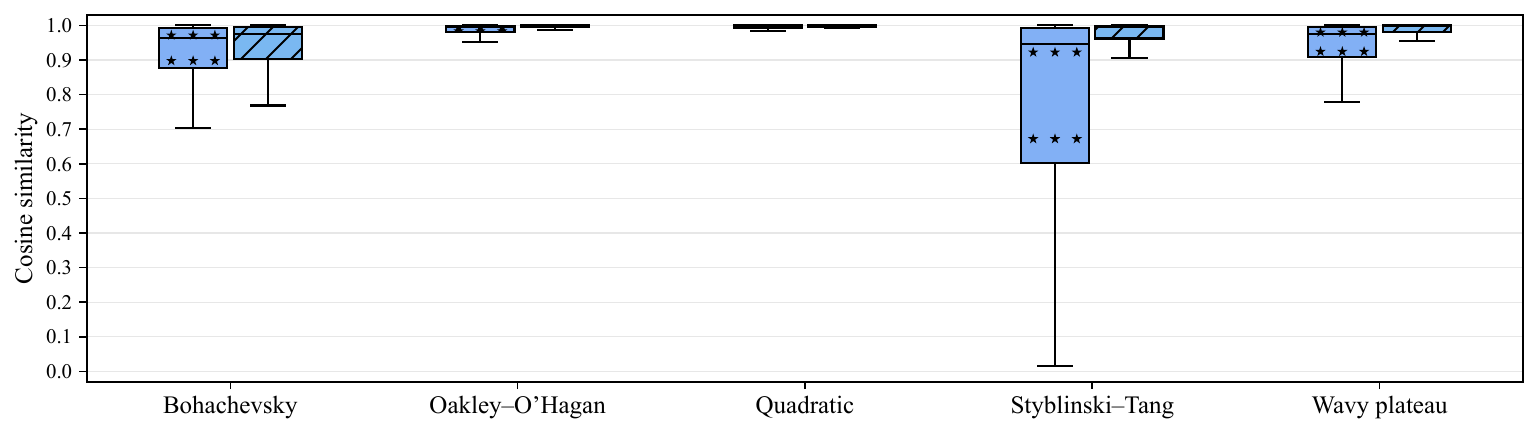}
        \label{fig:box_cos_appex_vs_jkonet}
    \end{subfigure}

    \vspace{1em}

     % --- Bottom subfigure ---
    \begin{subfigure}{\textwidth}
        \centering
        \textbf{(c) Absolute diffusivity error}\\[0.3em]
        \includegraphics[width=\textwidth]{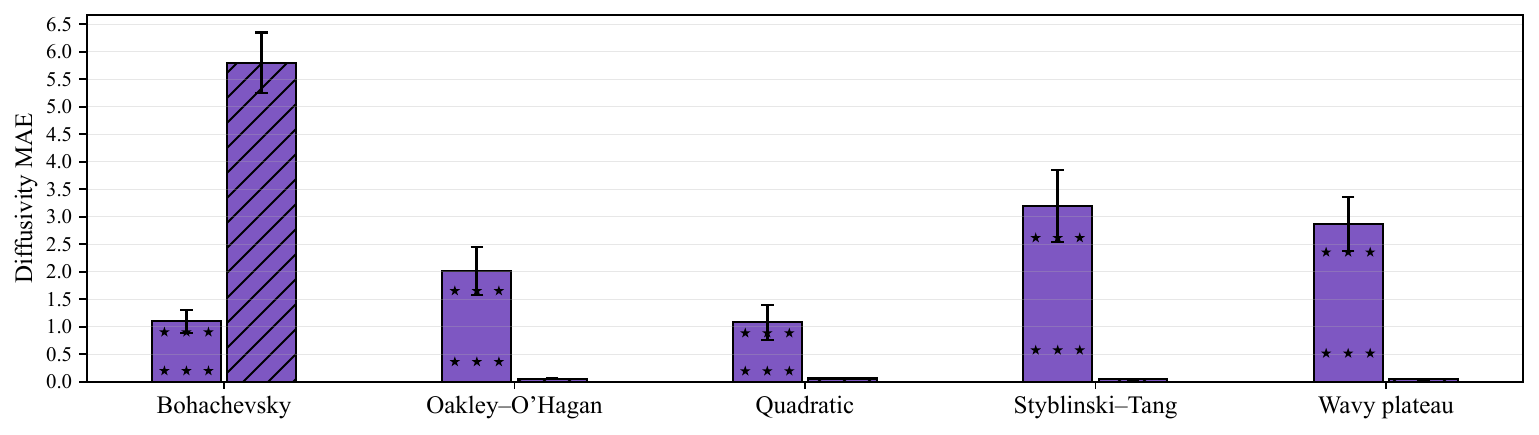}
        \label{fig:diff_mae_appex_jko}
    \end{subfigure}

    \vspace{1em}

    % --- Legend centered below ---
% Requires in preamble:
% \usetikzlibrary{patterns.meta,shapes.geometric}

\begin{tikzpicture}
  % sizes
  \def\w{1}        % rectangle width (cm)
  \def\h{0.4}      % rectangle height (cm)
  \def\dx{4.0}     % horizontal spacing (cm)
  \def\labgap{0.2} % label gap (cm)
  \def\starsize{3pt}

  % hatch style for nn-APPEX
  \def\hatchdist{4pt}
  \def\hatchwidth{0.3pt}
  \def\hatchangle{45}

  % JK0net*  (stars)
  \begin{scope}[shift={(0,0)}]
    \draw[black] (0,0) rectangle (\w,\h);
    % three filled stars across the swatch
    \foreach \x in {0.15,0.50,0.85}{
      \node[
        star, star points=5, star point ratio=2.25,
        draw=black, fill=black, minimum size=\starsize, inner sep=0pt
      ] at (\x,0.5*\h) {};
    }
    \node[anchor=west] at (\w+\labgap,\h/2) {\jkonetstar};
  \end{scope}

  % nn-APPEX (hatched)
  \begin{scope}[shift={(\dx,0)}]
    \path[
      pattern = {Lines[angle=\hatchangle, distance=\hatchdist, line width=\hatchwidth]},
      pattern color = black
    ] (0,0) rectangle (\w,\h);
    \draw[black] (0,0) rectangle (\w,\h);
    \node[anchor=west] at (\w+\labgap,\h/2) {\nappex};
  \end{scope}
\end{tikzpicture}

    \caption{The ability of the variational method \jkonetstar{} and \nappex{} (ours) to infer the gradient-flow drift is evaluated across five potentials using (a) normalized absolute error (lower is better) and (b) cosine similarity (higher is better), and their ability to infer the diffusivity is evaluated using (c) absolute error (lower is better). Both methods are given samples from three distinct marginals, such that the initial distribution is a Gaussian mixture model with randomly initialized components. The plots aggregated from $10$ seeds show that our method, \nappex{}, achieves better estimation of the drift and diffusion in this setting, except for the Bohachevsky potential.}
    \label{fig:boxplots_jko_appex}
\end{figure*}

\begin{table}[h]
    \centering
    \caption{Performance comparison of \sbirr{} and \nappex{} methods across various benchmark potential functions, reporting mean and median error metrics.}
    \label{tab:performance_comparison_mean_median}
    \begin{tabular}{llll}
        \toprule
        \textbf{Potential} & \textbf{Method} & \textbf{Mean} & \textbf{Median} \\
        \midrule
        Bohachevsky & \sbirr{} & 0.9135 & 0.5804 \\
        \cmidrule{2-4}
         & \nappex{} & 0.6749 & 0.4561 \\
        \midrule
        Oakley-O'Hagan & \sbirr{} & 0.5783 & 0.1786 \\
        \cmidrule{2-4}
         & \nappex{} & 0.1463 & 0.0839 \\
        \midrule
        Quadratic & \sbirr{} & 1.1515 & 0.1911 \\
        \cmidrule{2-4}
         & \nappex{} & 0.1261 & 0.0739 \\
        \midrule
        Styblinski-Tang & \sbirr{} & 0.5141 & 0.2245 \\
        \cmidrule{2-4}
         & \nappex{} & 0.4126 & 0.1509 \\
        \midrule
        Wavy Plateau & \sbirr{} & 2.1761 & 0.2107 \\
        \cmidrule{2-4}
         & \nappex{} & 0.9842 & 0.1960 \\
        \bottomrule
    \end{tabular}
\end{table}

%\newpage

\subsection{Comparing \appex{} and \nappex{}}

As a final experiment, we compare \appex{} and \nappex{} on a linear SDE problem in order to ensure \nappex{} is fully expressive and can compete with \appex{} in the linear domain for which \appex{} was designed. Results are displayed in Table \ref{tab:APPEX_vs_nnAPPEX}. As expected, \appex{} does quite a bit better than \nappex{} in terms of drift estimation since we have a closed form MLE estimate. \appex{}'s normalized MAE score is significantly better, while their cosine similarities are comparable, pointing to the fact that \nappex{} may do poorly in low drift regions. Interestingly, \nappex{} does slightly better in terms of diffusion estimation ($0.055$ vs $0.074$).

\begin{table}[h!]
\caption{Inference results for \nappex{} and \appex{} given the quadratic potential setting (see Appendix \ref{sec:potentials}). Both methods are given three marginals with Gaussian mixture model initialization.}
\label{tab:APPEX_vs_nnAPPEX}
    \centering
    \begin{tabular}{lccc}
        \toprule
        \textbf{Method} & \textbf{Normalized MAE} & \textbf{Cosine Similarity} & \textbf{Diffusivity error} \\
        \midrule
        \nappex{}       & 0.1550          & 0.9925          & \textbf{0.055} \\
        \appex{} (linear) & \textbf{0.0135} & \textbf{0.9996} & 0.074 \\
        \bottomrule
    \end{tabular}
\end{table}

\end{document}